\documentclass[12pt,a4paper]{amsart}
\usepackage{amssymb}
\usepackage{hyperref}

\calclayout

\newcommand{\+}{\protect\nobreakdash-}

\renewcommand{\:}{\colon}

\newcommand{\rarrow}{\longrightarrow}

\newcommand{\ot}{\otimes}

\newcommand{\bu}{{\text{\smaller\smaller$\scriptstyle\bullet$}}}

\DeclareMathOperator{\Hom}{Hom}
\DeclareMathOperator{\Ext}{Ext}

\DeclareMathOperator{\Spec}{Spec}

\DeclareMathOperator{\Fil}{\mathsf{Fil}}
\DeclareMathOperator{\Com}{\mathsf{Com}}

\newcommand{\Modr}{{\operatorname{\mathsf{Mod--}}}}
\newcommand{\Modl}{{\operatorname{\mathsf{--Mod}}}}

\newcommand{\bb}{\mathsf b}

\newcommand{\sA}{\mathsf A}
\newcommand{\sC}{\mathsf C}
\newcommand{\sD}{\mathsf D}
\newcommand{\sF}{\mathsf F}
\newcommand{\sI}{\mathsf I}
\newcommand{\sK}{\mathsf K}
\newcommand{\sL}{\mathsf L}
\newcommand{\sM}{\mathsf M}
\newcommand{\sN}{\mathsf N}
\newcommand{\sS}{\mathsf S}
\newcommand{\sT}{\mathsf T}

\newcommand{\Sets}{\mathsf{Sets}}
\newcommand{\GF}{\mathsf{GF}}
\newcommand{\GI}{\mathsf{GI}}

\newcommand{\boZ}{\mathbb Z}
\newcommand{\boQ}{\mathbb Q}

\newcommand{\p}{\mathfrak p}

\newcommand{\Section}[1]{\bigskip\section{#1}\medskip}
\theoremstyle{plain}
\newtheorem{thm}{Theorem}[section]

\newtheorem{lem}[thm]{Lemma}
\newtheorem{prop}[thm]{Proposition}
\newtheorem{cor}[thm]{Corollary}
\theoremstyle{definition}

\newtheorem{rem}[thm]{Remark}

\begin{document}

\title{Resolutions as directed colimits}

\author{Leonid Positselski}

\address{Institute of Mathematics, Czech Academy of Sciences \\
\v Zitn\'a~25, 115~67 Praha~1, Czech Republic}

\email{positselski@math.cas.cz}

\begin{abstract}
 A general principle suggests that ``anything flat is a directed colimit
of countably presentable flats''.
 In this paper, we consider resolutions and coresolutions of modules
over a countably coherent ring $R$ (e.~g., any coherent ring or any
countably Noetherian ring).
 We show that any $R$\+module of flat dimension~$n$ is a directed
colimit of countably presentable $R$\+modules of flat dimension
at most~$n$, and any flatly coresolved $R$\+module is a directed colimit
of countably presentable flatly coresolved $R$\+modules.
 If $R$ is a countably coherent ring with a dualizing complex, then any
F\+totally acyclic complex of flat $R$\+modules is a directed colimit
of F\+totally acyclic complexes of countably presentable flat
$R$\+modules.
 The proofs are applications of an even more general category-theoretic
principle going back to an unpublished 1977 preprint of Ulmer.
 Our proof of the assertion that every Gorenstein-flat module over
a countably coherent ring is a directed colimit of countably presentable
Gorenstein-flat modules uses a different technique, based on results
of \v Saroch and \v St\!'ov\'\i\v cek.
 We also discuss totally acyclic complexes of injectives and
Gorenstein-injective modules, obtaining various cardinality estimates
for the accessibility rank under various assumptions.
\end{abstract}

\maketitle

\tableofcontents

\section*{Introduction}
\medskip

 The classical Govorov--Lazard theorem~\cite{Gov,Laz} tells that all
flat modules are directed colimits of finitely generated projective
modules.
 What about modules of flat dimension~$n$, for a fixed $n\ge1$? 

 Any module of projective dimension~$1$, over an arbitrary ring,
is a directed colimit of finitely presentable modules of projective
dimension at most~$1$.
 Indeed, such a module is a cokernel of an injective morphism of
free modules, and such a morphism is a directed union of injective
morphisms of finitely generated free modules.
 This fact is mentioned, and a generalization to projective
dimension~$n$ is discussed, in the paper~\cite[Section~4]{BH};
see specifically~\cite[Theorem~4.6]{BH}.

 According to~\cite[Theorem~3.5]{AT}, any module of flat dimension~$1$
over a commutative integral domain is a directed colimit of finitely
presentable modules of projective dimension at most~$1$.
 Nevertheless, over commutative Noetherian local rings already,
a module of flat dimension~$1$ need not be a directed colimit of
(finitely generated) modules of projective dimension at most~$1$
\,\cite[Example~8.5]{BH}, \cite[Theorem~B]{HG}.
 A further discussion can be found in the book~\cite[Section~9.2]{GT}.

 In this paper we show that, over any right countably coherent ring,
any right module of flat dimension~$n$ is a directed colimit of
countably presentable modules of flat dimension at most~$n$
(see Corollary~\ref{finite-flat-dimension-cor}).
 Here a ring $R$ is said to be \emph{right countably coherent} if every
finitely generated right ideal in $R$ is countably presentable as
a right $R$\+module.
 In other words, in the language of~\cite[Chapter~2]{AR}, the category
of right $R$\+modules of flat dimension~$\le n$ is
\emph{$\aleph_1$\+accessible}, and its \emph{$\aleph_1$\+presentable
objects} are the countably presentable modules of flat
dimension~$\le n$ (i.~e., countably presentable as objects of
the category of all modules).
 Why is this interesting?

 From our perspective, the significance of flat modules with small
cardinalities of generators and relations lies in the fact that such
modules have finite projective dimensions.
 In fact, any countably presentable flat module has projective
dimension at most~$1$ \,\cite[Corollary~2.23]{GT}.
 More generally, any flat module with less than~$\aleph_m$
generators and relations has projective dimension at most~$m$
\,\cite[Proposition~5.3]{Jen} (see also our
Corollary~\ref{aleph-m-presentable-flat-module-cor}).

 There has been a stream of research about \emph{deconstructibility}
properties of various classes of modules and abelian/exact category
objects, using the \emph{Hill lemma}~\cite[Theorem~1]{Hil},
\cite[Theorem~2.1]{FL}, \cite[Theorem~6]{StT}, \cite[Theorem~7.10]{GT},
\cite[Theorem~2.1]{Sto} as one of the main technical tools.
 The importance of deconstructibility is explained by
the Eklof lemma~\cite[Lemma~1]{ET}, \cite[Lemma~6.2]{GT} and
the Eklof--Trlifaj theorem~\cite[Theorems~2 and~10]{ET},
\cite[Theorem~6.11 and Corollary~6.14]{GT}.

 In this context, a natural question is whether every module of
projective dimension~$n$ is a direct summand of a module filtered by
modules of projective dimension at most~$n$ admitting a resolution
by finitely generated projective modules.
 The answer is negative for $n=1$ and commutative Noetherian local
rings already~\cite[Theorem~8.6 or Lemma~9.1]{BH},
\cite[Theorem~3.14]{HG0}, \cite[Theorem~A]{HG}.
 So a module of projective dimension~$1$ need not be a direct summand
of a module filtered by finitely generated modules of projective
dimension~$1$.

 On the other hand, over a right countably Noetherian ring, any right
module of projective dimension~$n$ is filtered by countably generated
modules of projective dimension at most~$n$
\,\cite[Corollaire~II.3.2.5]{RG}, \cite[Proposition~4.1]{AEJO}.
 Here a ring $R$ is called \emph{right countably Noetherian} if every
right ideal in $R$ is countably generated.
 A far-reaching generalization of this result can be found
in~\cite[Theorem~3.4]{SlT}.

 An approach to accessibility based on deconstructibility is possible,
but its results may be suboptimal.
 Let $\kappa$~be a regular cardinal and $\sS$ be a set of
$\kappa$\+presentable modules over a ring~$R$.
 Then it follows from the Hill lemma that every module filtered by $\sS$
is a $\kappa$\+directed union of $\kappa$\+presentable modules
filtered by~$\sS$.
 Arguing in this way, and using the mentioned result from~\cite{SlT}
together with purification considerations~\cite[Lemma~1
and Proposition~2]{BBE}, one can show that any $R$\+module of flat
dimension~$n$ is a directed colimit of $\kappa$\+presentable
$R$\+modules of flat dimension at most~$n$, where $\kappa$~is
an uncountable regular cardinal greater than the cardinality of~$R$.
 This is not as good as the result of our
Corollary~\ref{finite-flat-dimension-cor}, which gives
$\kappa=\aleph_1$ for modules over countably coherent rings.

 The main techniques presented in this paper are category-theoretic in
nature, based on a general principle going back to an unpublished
1977~preprint of Ulmer~\cite{Ulm} and exemplified by the Pseudopullback
Theorem of Raptis and Rosick\'y~\cite[Proposition~3.1]{CR},
\cite[Theorem~2.2]{RR}.
 Another exposition can be found in the recent preprint~\cite{Pacc}
by the present author. {\uchyph=0\par}

 With these methods, we can and do treat coresolutions on par with
resolutions.
 In particular, one says that an $R$\+module $M$ is \emph{flatly
coresolved} if there exists an exact sequence of $R$\+modules
$0\rarrow M\rarrow F^0\rarrow F^1\rarrow F^2\rarrow\dotsb$ with
flat $R$\+modules~$F^i$.
 We show that, over a countably right coherent ring $R$, any flatly
coresolved right $R$\+module is a directed colimit of countably
presentable flatly coresolved right $R$\+modules
(see Corollary~\ref{flatly-coresolved-modules-cor}).

 The latter result brings us closer to the presently popular subject
of Gorenstein homological algebra.
 Approaching it with our category-theoretic methods, we show that,
over a right countably coherent ring $R$ with a dualizing complex, any
F\+totally acyclic complex of flat right $R$\+modules (in the sense
of~\cite[Section~2]{EFI}) is an $\aleph_1$\+directed colimit of
F\+totally acyclic complexes of countably presentable flat modules
(see Theorem~\ref{dualizing-complex-F-totally-acyclics-theorem}).
 Consequently, any Gorenstein-flat right $R$\+module is a directed
colimit of countably presentable Gorenstein-flat modules.
 Furthermore, any countably presentable Gorenstein-flat right
$R$\+module is a direct summand of a (Gorenstein-flat) module
admitting an F\+totally acyclic two-sided resolution by countably
presentable flat $R$\+modules
(Corollary~\ref{dualizing-complex-F-totally-acyclics-direct-summands}).
 The same results apply to commutative Noetherian rings with at most
countable spectrum (instead of a dualizing complex).

 A more powerful approach is based on difficult results of
the paper~\cite{SarSt}, or specifically~\cite[Theorems~4.9
and~4.11(4)]{SarSt}.
 It allows to prove that, over any right countably coherent ring,
any Gorenstein-flat right module is an $\aleph_1$\+directed colimit
of countably presentable Gorenstein-flat modules
(Theorem~\ref{countably-coherent-Gorenstein-flats-theorem}).

 We also consider Gorenstein-injective modules over a left Noetherian
ring $R$ with a dualizing complex.
 Here we prove that, if the cardinality of $R$ is smaller than~$\kappa$,
then any totally acyclic complex of injective left $R$\+modules is
a $\kappa$\+directed colimit of totally acyclic complexes of injective
modules of cardinality less than~$\kappa$
(see Theorem~\ref{dualizing-complex-tot-acycl-of-inj-theorem}).
 Consequently, every Gorenstein-injective left $R$\+module is
a directed colimit of Gorenstein-injective modules of cardinality
less than~$\kappa$.
 Furthermore, any Gorenstein-injective left $R$\+module of
cardinality less than~$\kappa$ is a direct summand of
a (Gorenstein-injective) module admitting a totally acyclic two-sided
resolution by injective $R$\+modules of cardinality less
than~$\kappa$\,
(Corollary~\ref{dualizing-complex-tot-acycl-of-inj-direct-summand}).

 More generally, let us say that an $R$\+module $M$ is
\emph{injectively resolved} if there exists an exact sequence of
$R$\+modules $\dotsb\rarrow J_2\rarrow J_1\rarrow J_0\rarrow M\rarrow0$
with injective $R$\+modules~$J_i$.
 We show that, over a left Noetherian ring $R$, any injectively
resolved left $R$\+module is a directed colimit of injectively
resolved $R$\+modules of cardinality not exceeding $\aleph_0$ plus
the cardinality of~$R$.

 Let us say a few more words about the proofs.
 The main arguments in this paper are based on a very general
category-theoretic principle going back to~\cite{Ulm} and rediscovered
in~\cite{Pacc}.
 The Pseudopullback Theorem~\cite[Proposition~3.1]{CR},
\cite[Theorem~2.2]{RR} is one of the instances of this principle, and
it is essentially sufficient for our purposes.
 In full generality, the principle involves a regular cardinal~$\kappa$
and a smaller infinite cardinal $\lambda<\kappa$, and the claim is that
the class of all $\kappa$\+accessible categories with directed colimits
of $\lambda$\+indexed chains is stable under many category-theoretic
constructions, including the pseudopullbacks, the inserters, and
the equifiers.
 Moreover, there is a good control over the full subcategories of
$\kappa$\+presentable objects in the categories arising under such
constructions; this is crucial for applications.

 A general principle that \emph{anything flat is a directed colimit of
countably presentable flats} is mentioned in the abstract.
 This is illustrated by the results of
the preprints~\cite[Sections~10.2, 10.4, and~10.5]{Pacc},
\cite[Sections~2\+-4]{PS6}, \cite[Sections~3, 4, 10, and~11]{Pflcc},
and~\cite[Sections~4, 6, and~7]{Pcor}.
 This is also confirmed by the results of the present paper, such as
Corollaries~\ref{finite-flat-dimension-cor}
and~\ref{flatly-coresolved-modules-cor}, and
Theorem~\ref{countably-coherent-Gorenstein-flats-theorem}.

 Sections~\ref{accessible-and-acyclic-secn}\+-%
\ref{deconstructibility-secn} contain various kinds of preliminary
material and preliminary discussions.
 In Sections~\ref{two-sided-resolutions-secn}\+-%
\ref{flatly-coresolved-secn} we discuss modules with two-sided
resolutions by modules from accessible classes and two particular cases,
modules of bounded flat dimension and flatly coresolved modules.
 In the subsequent
Sections~\ref{dualizing-and-F-totally-acyclic-secn}\+-%
\ref{gorenstein-flats-secn} we study F\+totally acyclic complexes
and Gorenstein-flat modules.
 In the next Sections~\ref{finite-injective-dimension-secn}\+-%
\ref{injectively-resolved-secn} we construct modules of bounded
injective dimension and injectively resolved modules as directed
colimits.
 The final Sections~\ref{left-noetherian-secn}\+-%
\ref{full-generality-secn} are dedicated to accessibility properties
of totally acyclic complexes of injectives and
Gorenstein-injective modules.

\subsection*{Acknowledgement}
 This work was inspired by a conversation with Michal Hrbek, and took
its present form as a result of extensive consultations and
communications with Michal Hrbek, Jan Trlifaj, Giovanna Le~Gros,
Dolors Herbera, Silvana Bazzoni, and Jan \v St\!'ov\'\i\v cek.
 Liran Shaul suggested the idea of using dualizing complexes for
characterizing total resolutions.
 I~wish to thank all the mentioned people for their help.
 Quite separately, I~want to express my gratitude to
Jan \v St\!'ov\'\i\v cek, who brought~\cite[Theorems~4.9
and~4.11(4)]{SarSt} to my attention, and suggested to include
the topic of Gorenstein-injective modules.
 Lemma~\ref{functorial-totally-injective-lemma} and
Corollary~\ref{Noeth-Gorenst-injectives-rho-dir-colim-closed}
are also due to Jan.
 Thanks are also due to an anonymous referee for careful reading
of the manuscript and suggesting a number of small corrections,
as well as the reference~\cite{Jen} in connection with
Corollary~\ref{aleph-m-presentable-flat-module-cor}.
 The author is supported by the GA\v CR project 23-05148S and
the Czech Academy of Sciences (RVO~67985840).
{\uchyph=0\par}

\Section{Accessible Categories and Acyclic Complexes}
\label{accessible-and-acyclic-secn}

 We use the book~\cite{AR} as the main reference source on accessible
categories.
 In particular, we refer to~\cite[Definition~1.4, Theorem and
Corollary~1.5, Definition~1.13(1), and Remark~1.21]{AR} for
the discussion of \emph{$\kappa$\+directed posets} vs.\
\emph{$\kappa$\+filtered small categories} and, accordingly,
\emph{$\kappa$\+directed} vs.\ \emph{$\kappa$\+filtered colimits}.
 Here $\kappa$~denotes a regular cardinal.
 Let us just mention that a poset $\Xi$ is said to be
\emph{$\kappa$\+directed} if every subset of cardinality less
than~$\kappa$ has an upper bound in~$\Xi$.

 Let $\sK$ be a category with $\kappa$\+directed (equivalently,
$\kappa$\+filtered) colimits.
 An object $S\in\sK$ is said to be
\emph{$\kappa$\+presentable}~\cite[Definition~1.13(2)]{AR}
(or ``$<\kappa$\+presented'' in the traditional module-theoretic
terminology) if the functor $\Hom_\sK(S,{-})\:\sK\rarrow\Sets$
preserves $\kappa$\+directed colimits.
 We will denote the full subcategory of $\kappa$\+presentable objects
of $\sK$ by $\sK_{<\kappa}\subset\sK$.

 The category $\sK$ is called
\emph{$\kappa$\+accessible}~\cite[Definition~2.1]{AR} if there
is a \emph{set} of $\kappa$\+presentable objects $\sS\subset\sK$
such that every object of $\sK$ is a $\kappa$\+directed colimit of
objects from~$\sS$.
 If this is the case, then the $\kappa$\+presentable objects of $\sK$
are precisely all the retracts of the objects from~$\sS$.
 A $\kappa$\+accessible category $\sK$ is called
\emph{locally $\kappa$\+presentable} if all colimits exist in~$\sK$
\,\cite[Definition~1.17 and Theorem~1.20]{AR}.

 In the standard terminology, $\aleph_0$\+presentable objects are called
\emph{finitely presentable}~\cite[Definition~1.1]{AR}; and we will
call $\aleph_1$\+presentable objects \emph{countably presentable}.
 $\aleph_0$\+accessible categories are called \emph{finitely
accessible}~\cite[Remark~2.2(1)]{AR}, and locally $\aleph_0$\+presentable
categories are called \emph{locally finitely
presentable}~\cite[Definition~1.9 and Theorem~1.11]{AR}.
{\hbadness=1050\par}

 Given a category\/ $\sK$ with $\kappa$\+directed colimits and a class
of objects $\sT\subset\sK$, we denote by $\varinjlim_{(\kappa)}\sT
\subset\sK$ the class of all $\kappa$\+directed colimits of objects
from~$\sT$.
 In the case of $\kappa=\aleph_0$, we will simply write $\varinjlim\sT$
instead of $\varinjlim_{(\aleph_0)}\sT$.

 The following proposition is essentially well-known.

\begin{prop} \label{accessible-subcategory}
 Let $\kappa$~be a regular cardinal and\/ $\sK$ be
a\/ $\kappa$\+accessible category.
 Let\/ $\sT\subset\sK$ be a set of $\kappa$\+presentable objects.
 Then the full subcategory\/ $\varinjlim_{(\kappa)}\sT\subset\sK$ is
closed under $\kappa$\+directed colimits in\/~$\sK$.
 The category\/ $\varinjlim_{(\kappa)}\sT$ is $\kappa$\+accessible,
and the $\kappa$\+presentable objects of\/ $\varinjlim_{(\kappa)}\sT$
are precisely all the retracts of objects from\/~$\sT$.
 An object $K\in\sK$ belongs to\/ $\varinjlim_{(\kappa)}\sT$ if and
only if, for every object $S\in\sK_{<\kappa}$, any morphism $S\rarrow K$
in\/ $\sK$ factorizes through an object from\/~$\sT$.
\end{prop}

\begin{proof}
 In the particular case of finitely accessible additive categories,
versions of this result were discussed in~\cite[Proposition~2.1]{Len},
\cite[Section~4.1]{CB}, and~\cite[Proposition~5.11]{Kra}.
 (The terminology ``finitely presented categories'' was used
in~\cite{CB,Kra} for what we call finitely accessible categories.)
 The key step is to prove the ``if'' part of the last assertion; then
the remaining arguments are easy.
\end{proof}

\begin{prop} \label{product-proposition}
 Let $\kappa$~be a regular cardinal, and let $(\sK_\xi)_{\xi\in\Xi}$ be
a family of $\kappa$\+accessible categories, indexed by a set\/ $\Xi$
of cardinality less than~$\kappa$.
 Then the Cartesian product category\/ $\sK=\prod_{\xi\in\Xi}\sK_\xi$ is
also $\kappa$\+accessible.
 The $\kappa$\+presentable objects of\/ $\sK$ are precisely all
the collections of objects $(S_\xi\in\sK_\xi)_{\xi\in\Xi}$ such that
the object $S_\xi$ is $\kappa$\+presentable in\/ $\sK_\xi$ for every
$\xi\in\Xi$.
\end{prop}

\begin{proof}
 This is a corrected version of~\cite[proof of Proposition~2.67]{AR}.
 We refer to~\cite[Proposition~2.1]{Pacc} for the details.
\end{proof}

 In the next two theorems, we consider a regular cardinal~$\kappa$
and a smaller infinite cardinal $\lambda<\kappa$ (so $\kappa$~is
necessarily uncountable).
 The cardinal~$\lambda$ is viewed as an ordinal, and directed colimits
of $\lambda$\+indexed chains are considered.
 Here a \emph{$\lambda$\+indexed chain} in $\sK$ is a directed diagram
of the form $(K_i\to K_j)_{0\le i<j<\lambda}$.

 Let $\sK_1$, $\sK_2$, and $\sL$ be three categories, and let
$\Phi_1\:\sK_1\rarrow\sL$ and $\Phi_2\:\sK_2\rarrow\sL$ be two functors.
 The \emph{pseudopullback} $\sC$ of the pair of functors $\Phi_1$
and $\Phi_2$ is defined as the category of triples $(K_1,K_2,\theta)$,
where $K_1\in\sK_1$ and $K_2\in\sK_2$ are two objects, and
$\theta\:\Phi_1(K_1)\simeq\Phi_2(K_2)$ is an isomorphism in~$\sL$.

\begin{thm} \label{pseudopullback-theorem}
 Let $\kappa$~be a regular cardinal and $\lambda<\kappa$ be a smaller
infinite cardinal.
 Let\/ $\sK_1$, $\sK_2$, and\/ $\sL$ be three $\kappa$\+accessible
categories where colimits of $\lambda$\+indexed chains exist.
 Assume that two functors\/ $\Phi_1\:\sK_1\rarrow\sL$ and\/
$\Phi_2\:\sK_2\rarrow\sL$ preserve $\kappa$\+directed colimits
and colimits of $\lambda$\+indexed chains, and take
$\kappa$\+presentable objects to $\kappa$\+presentable objects.
 Then the pseudopullback\/ $\sC$ is a $\kappa$\+accessible category.
 The $\kappa$\+presentable objects of\/ $\sC$ are precisely all
the triples $(S_1,S_2,\theta)\in\sC$ such that the object $S_1$ is
$\kappa$\+presentable in\/ $\sK_1$ and the object $S_2$ is
$\kappa$\+presentable in\/~$\sK_2$.
\end{thm}

\begin{proof}
 This assertion, going back to~\cite[Remark~3.2(I), Theorem~3.8,
Corollary~3.9, and Remark~3.11(II)]{Ulm}, can be found
in~\cite[Pseudopullback Theorem~2.2]{RR} with the proof
in~\cite[Proposition~3.1]{CR}.
 See also~\cite[Corollary~5.1]{Pacc}.
\end{proof}

 Let $\sK$ and $\sL$ be two categories, and let $\Phi_1$, $\Phi_2\:\sK
\rightrightarrows\sL$ be two parallel functors.
 The \emph{isomorpher} $\sC$ of the pair of functors $\Phi_1$ and
$\Phi_2$ is defined as the category of pairs $(K,\theta)$, where
$K\in\sK$ is an object and $\theta\:\Phi_1(K)\simeq\Phi_2(K)$ is
an isomorphism in~$\sL$.

\begin{thm} \label{isomorpher-theorem}
 Let $\kappa$~be a regular cardinal and $\lambda<\kappa$ be a smaller
infinite cardinal.
 Let\/ $\sK$ and\/ $\sL$ be two $\kappa$\+accessible categories where
colimits of $\lambda$\+indexed chains exist.
 Assume that two functors\/ $\Phi_1$ and\/ $\Phi_2\:
\sK\rightrightarrows\sL$ preserve $\kappa$\+directed colimits and
colimits of $\lambda$\+indexed chains, and take $\kappa$\+presentable
objects to $\kappa$\+presentable objects.
 Then the isomorpher\/ $\sC$ is a $\kappa$\+accessible category.
 The $\kappa$\+presentable objects of\/ $\sC$ are precisely all
the pairs $(S,\theta)\in\sC$ such that the object $S$ is
$\kappa$\+presentable in\/~$\sK$.
\end{thm}

\begin{proof}
 This is essentially an equivalent version of
Theorem~\ref{pseudopullback-theorem}, as explained
in~\cite[Remark~5.2]{Pacc}.
 The reference to~\cite[Remark~3.2(I), Theorem~3.8, Corollary~3.9,
and Remark~3.11(II)]{Ulm} is also applicable.
\end{proof}

 Let $R$ be an associative ring.
 We will denote by $\Modr R$ the category of right $R$\+modules and
by $R\Modl$ the category of left $R$\+modules.
 The abelian category $\Modr R$ is locally finitely presentable, hence
locally $\kappa$\+presentable for every regular cardinal~$\kappa$
\,\cite[Remark~1.20]{AR}.
 The $\kappa$\+presentable objects of $\Modr R$ are precisely all
the $R$\+modules with less than~$\kappa$ generators and less
than~$\kappa$ relations, i.~e., in other words, the cokernels of
morphisms of free $R$\+modules with less than~$\kappa$ generators.

 For any additive category $\sA$, we denote by $\Com(\sA)$ the category
of (unbounded) complexes in~$\sA$.

\begin{lem} \label{abelian-complexes-lemma}
 For any associative ring $R$, the following assertions hold: \par
\textup{(a)} The abelian category\/ $\Com(\Modr R)$ of complexes of
right $R$\+modules is locally finitely presentable.
 Consequently, this category is locally $\kappa$\+presentable for any
regular cardinal~$\kappa$. \par
\textup{(b)} The finitely presentable objects of\/ $\Com(\Modr R)$ are
precisely all the \emph{bounded} complexes of finitely presentable
$R$\+modules. \par
\textup{(c)} For every uncountable regular cardinal~$\kappa$,
the $\kappa$\+presentable objects of\/ $\Com(\Modr R)$ are precisely
all the (unbounded) complexes of $\kappa$\+presentable $R$\+modules.
\hbadness=1725
\end{lem}

\begin{proof}
 Essentially, complexes are modules over a suitable ``ring with many
objects'' (viz., the objects are indexed by the integers $n\in\boZ$).
 As usual, the results about modules over rings are applicable to
modules over rings with many objects, and provide the assertions of
the lemma.
 Alternatively, part~(c) can be deduced by applying an additive
version of~\cite[Theorem~1.2]{Hen} or~\cite[Theorem~6.2]{Pacc}.
\end{proof}

 The \emph{category of epimorphisms of right $R$\+modules} has
epimorphisms of right $R$\+modules $L\rarrow M$ as objects and
commutative squares $L'\rarrow M'\rarrow M''$, \ $L'\rarrow L''
\rarrow M''$ (with epimorphisms $L'\rarrow M'$ and $L''\rarrow M''$)
as morphisms.

\begin{lem} \label{epimorphism-lemma}
 For any ring $R$ and every regular cardinal~$\kappa$, the category of
epimorphisms of right $R$\+modules is $\kappa$\+accessible.
 The $\kappa$\+presentable objects of this category are
the epimorphisms of $\kappa$\+presentable right $R$\+modules.
\end{lem}

\begin{proof}
 This is a particular case of~\cite[Lemma~10.7]{Pacc}.
\end{proof}

 We will say that an associative ring $R$ is \emph{right\/
$<\kappa$\+coherent} if, for every $\kappa$\+pre\-sentable right
$R$\+module $S$, any submodule in $S$ having less than~$\kappa$
generators is $\kappa$\+presentable.
 Equivalently, $R$ is right $<\kappa$\+coherent if and only if every
right ideal in $R$ with less than~$\kappa$ generators is
$\kappa$\+presentable as a right $R$\+module, and if and only if
every finitely generated right ideal in $R$ is $\kappa$\+presentable.
 We will call right $<\aleph_1$\+coherent rings \emph{right countably
coherent}.

\begin{cor} \label{short-exact-sequences-cor}
 For any regular cardinal~$\kappa$ and any right\/ $<\kappa$\+coherent
ring $R$, the category of short exact sequences of right $R$\+modules
is $\kappa$\+accessible.
 The $\kappa$\+presentable objects of this category are
the short exact sequences of $\kappa$\+presentable right $R$\+modules.
\end{cor}

\begin{proof}
 The first assertion is a restatement of the first assertion of
Lemma~\ref{epimorphism-lemma} and holds for any ring~$R$; but one
needs the $<\kappa$\+coherence assumption in order to obtain
the second assertion of the corollary from the second assertion of
the lemma (cf.~\cite[Corollary~10.13]{Pacc}).
\end{proof}

\begin{prop} \label{acyclic-complexes-prop}
 For any uncountable regular cardinal~$\kappa$ and any right\/
$<\kappa$\+coherent ring $R$, the category of (unbounded) acyclic
complexes of right $R$\+modules is $\kappa$\+accessible.
 The $\kappa$\+presentable objects of this category are
the acyclic complexes of $\kappa$\+presentable right $R$\+modules.
\end{prop}

\begin{proof}
 The argument is similar to~\cite[proof of Corollary~10.14]{Pacc}.
 Notice that an acyclic complex of modules $C^\bu$ is the same thing
as a collection of short exact sequences of modules $0\rarrow K^n
\rarrow C^n\rarrow M^n\rarrow0$, \,$n\in\boZ$, together with
an isomorphism of modules $M^n\simeq K^{n+1}$ for every $n\in\boZ$.
 This observation allows to construct the category of acyclic complexes
of $R$\+modules from the category of short exact sequences of
$R$\+modules using Cartesian products and the isomorpher construction.

 Specifically, put $\sL=\prod_{n\in\boZ}\Modr R$, and let $\sK$ be
the Cartesian product of the categories of short exact sequences of
right $R$\+modules taken over all $n\in\boZ$.
 Let $\Phi_1\:\sK\rarrow\sL$ be the functor assigning to a family of
short exact sequences $(0\to K^n\to L^n\to M^n\to0)_{n\in\boZ}$
the family of modules $(M^n)_{n\in\boZ}$, and let $\Phi_2$ be
the functor assigning to the same family of short exact sequences
the family of modules $(K^{n+1})_{n\in\boZ}$.
 Then the isomorpher category $\sC$ is equivalent to the desired
category of acyclic complexes of right $R$\+modules.

 The categories $\sK$ and $\sL$ are $\kappa$\+accessible by
Corollary~\ref{short-exact-sequences-cor}
and Proposition~\ref{product-proposition}, and
Theorem~\ref{isomorpher-theorem} is applicable (for $\lambda=\aleph_0$).
 The theorem tells that the category $\sC$ is $\kappa$\+accessible,
and provides the desired description of its full subcategory of
$\kappa$\+presentable objects.
\end{proof}

\Section{Modules of Small Presentability Rank as Small Directed
Colimits}

 The aim of this section is to prove the following proposition,
which is purported to complement the main results of this paper.
 Recall that an $R$\+module is said to be \emph{$\kappa$\+presentable}
in our (category-theoretic) terminology if it is the cokernel of
a morphism of free $R$\+modules with \emph{less than~$\kappa$}
generators.

\begin{prop} \label{aleph-m-presentable-directed-colimit-prop}
 Let $R$ be an associative ring and\/ $\sS$ be a set of finitely
presentable right $R$\+modules.
 Let $C\in\Modr R$ be an\/ $\aleph_m$\+presentable $R$\+module
belonging to\/ $\varinjlim\sS\subset\Modr R$ (where $m\ge0$ is
an integer).
 Let $D\in\sS^{\perp_{\ge1}}$ be a right $R$\+module such that\/
$\Ext^i_R(S,D)=0$ for all $S\in\sS$ and $i>0$.
 Then\/ $\Ext^i_R(C,D)=0$ for all $i>m$.
\end{prop}

\begin{lem} \label{kappa-presentable-directed-colimit-lemma}
 Let $R$ be an associative ring and\/ $\sS$ be a set of finitely
presentable $R$\+modules.
 Let $\kappa$~be a regular cardinal and $C$ be a $\kappa$\+presentable
$R$\+module belonging to $\varinjlim\sS$.
 Then $C$ is a direct summand of a directed colimit of modules from\/
$\sS$ indexed by a directed poset of cardinality less than\/~$\kappa$.
\end{lem}

\begin{proof}
 Let $\sT$ denote the class of all directed colimits of modules from
$\sS$ indexed by directed posets of cardinality less than~$\kappa$.
 Then, following~\cite[proof of Theorem~2.11(iv)\,$\Rightarrow$\,(i)
and Example~2.13(1)]{AR} (with $\lambda=\aleph_0$ and $\mu=\kappa$),
every module from $\varinjlim\sS$ is a $\kappa$\+directed colimit
of modules from~$\sT$.
 For a $\kappa$\+presentable module $C\in\varinjlim\sS$, it follows
that $C$ is a direct summand of a module from~$\sT$.
\end{proof}

 The following lemma can be found in~\cite[Th\'eor\`eme~4.2]{Jen}.

\begin{lem} \label{varprojlim-Ext-spectral-sequence-lemma}
 Let\/ $\Xi$ be a directed poset, $(S_\xi)_{\xi\in\Xi}$ be
a\/ $\Xi$\+indexed diagram of right $R$\+modules, and $D$ be a right
$R$\+module.
 Let\/ $\varprojlim_{\xi\in\Xi}^n$ denote the derived functors of\/
$\Xi$\+indexed limit of abelian groups.
 Then there is a spectral sequence
$$
 E_2^{pq}=\varprojlim\nolimits_{\xi\in\Xi}^p\Ext_R^q(S_\xi,D)
 \Longrightarrow E_\infty^{pq}=\mathrm{gr}^p\Ext_R^{p+q}
 (\varinjlim\nolimits_{\xi\in\Xi}S_\xi,\>D).
$$
\end{lem}

\begin{proof}
 Let $J^\bu$ be an injective coresolution of the $R$\+module $D$,
and let $B_\bu$ be the bar-complex of the diagram $(S_\xi)_{\xi\in\Xi}$.
 Consider the bicomplex of abelian groups $A^{pq}=\Hom_R(B_p,J^q)$.
 Let $T^\bu$ denote the total complex of the bicomplex $A^{\bu,\bu}$.

 Then one has $H_0(B_\bu)\simeq\varinjlim_{\xi\in\Xi}S_\xi$ and
$H_p(B_\bu)=0$ for all $p>0$ (because the directed colimit functors
are exact).
 Consequently, $H^0(A^{\bu,q})=H^0(\Hom_R(B_\bu,J^q))\simeq
\Hom_R(\varinjlim_{\xi\in\Xi}S_\xi,\>J^q)$ and
$H^p(A^{\bu,q})=H^p(\Hom_R(B_\bu,J^q))=0$ for all $q\ge0$ and $p>0$
(since $J^q$ is an injective $R$\+module).
 Hence a natural isomorphism $H^n(T^\bu)\simeq
H^n(\Hom_R(\varinjlim_{\xi\in\Xi}S_\xi,\>J^\bu))=
\Ext_R^n(\varinjlim_{\xi\in\Xi}S_\xi,\>D)$ for all $n\ge0$.

 On the other hand, for every $q\ge0$, one has
$H^q(A^{p,\bu})=\Ext_R^q(B_p,D)$.
 The complex $H^q(A^{\bu,\bu})$ is the cobar-complex computing
the derived functor of $\Xi$\+indexed limit
$\varprojlim_{\xi\in\Xi}^*\Ext_R^q(S_\xi,D)$, so one has
$H^pH^q(A^{\bu,\bu})=\varprojlim_{\xi\in\Xi}^p\Ext_R^q(S_\xi,D)$.
 Thus the spectral sequence $E_2^{pq}=H^pH^q(A^{\bu,\bu})
\Longrightarrow E_\infty^{pq}=\mathrm{gr}^pH^{p+q}(T^\bu)$ is
the desired one.
\end{proof}

\begin{proof}[Proof of
Proposition~\ref{aleph-m-presentable-directed-colimit-prop}]
 In view of Lemma~\ref{kappa-presentable-directed-colimit-lemma},
one can assume without loss of generality that
$C=\varinjlim_{\xi\in\Xi}S_\xi$, where $\Xi$ is a directed poset
of cardinality smaller than~$\aleph_m$ and $S_\xi\in\sS$ for
all $\xi\in\Xi$.
 Then the spectral sequence from
Lemma~\ref{varprojlim-Ext-spectral-sequence-lemma} degenerates to
a natural isomorphism $\Ext^n_R(C,D)\simeq\varprojlim^n_{\xi\in\Xi}
\Hom_R(S_\xi,D)$ for all $n\ge0$.
 It remains to recall that the derived functor of $\Xi$\+indexed
limit in the category of abelian groups has cohomological dimension
at most~$m$ \,\cite{Mit}.
\end{proof}

 The following corollary can be found in~\cite[Proposition~5.3]{Jen}.
 It is a partial generalization of~\cite[Corollary~2.23]{GT}.

\begin{cor} \label{aleph-m-presentable-flat-module-cor}
 Let $R$ be an associative ring and $F$ be an\/ $\aleph_m$\+presentable
flat $R$\+module.
 Then the projective dimension of $F$ does not exceed~$m$.
\end{cor}

\begin{proof}
 Let $\sS$ be the set of all finitely generated projective (or free)
$R$\+modules, $D$ be an arbitrary $R$\+module, and apply
Proposition~\ref{aleph-m-presentable-directed-colimit-prop}.
\end{proof}

\Section{Deconstructibility and Directed Colimits}
\label{deconstructibility-secn}

 In this section we discuss the deconstructibility-based approach to
accessibility.
 Both the deconstructible classes and the right $\Ext^1$\+orthogonal
classes to deconstructible classes are considered.

 Let $R$ be an associative ring, $F$ be an $R$\+module, and
$\alpha$~be an ordinal.
 An \emph{$\alpha$\+indexed filtration} of $F$ is a family of
submodules $F_\beta\subset F$, indexed by the ordinals
$0\le\beta\le\alpha$, satisfying the following conditions:
\begin{itemize}
\item $F_0=0$ and $F_\alpha=F$;
\item one has $F_\gamma\subset F_\beta$ for all $0\le\gamma\le\beta
\le\alpha$;
\item one has $F_\beta=\bigcup_{\gamma<\beta}F_\gamma$ for all limit
ordinals $\beta\le\alpha$.
\end{itemize}

 An $R$\+module $F$ endowed with an $\alpha$\+indexed filtration
$(F_\beta)_{0\le\beta\le\alpha}$ is said to be \emph{filtered by}
the quotient modules $F_{\beta+1}/F_\beta$, \,$0\le\beta<\alpha$.
 Given a class of $R$\+modules $\sS\subset\Modr R$, one says that
an $R$\+module $F$ is \emph{filtered by} $\sS$ if there exists
an ordinal~$\alpha$ and an $\alpha$\+indexed filtration on $F$
such that the quotient module $F_{\beta+1}/F_\beta$ is isomorphic
to a module from $\sS$ for every $0\le\beta<\alpha$.

 The class of all $R$\+modules filtered by $\sS$ is denoted by
$\Fil(\sS)\subset\Modr R$.
 A class of $R$\+modules $\sF$ is said to be
\emph{$\kappa$\+deconstructible} (for a regular cardinal~$\kappa$) if
$\sF=\Fil(\sS)$ for a set of $\kappa$\+presentable $R$\+modules~$\sS$.

\begin{prop} \label{deconstructible-as-directed-colimit}
 Let $R$ be an associative ring, $\kappa$~be a regular cardinal,
and\/ $\sS$ be a set of $\kappa$\+presentable $R$\+modules.
 Then any $R$\+module filtered by\/ $\sS$ is a $\kappa$\+directed
colimit of $\kappa$\+presentable $R$\+modules filtered by\/~$\sS$.
 In other words, for any $\kappa$\+deconstructible class of modules\/
$\sF$, all modules from\/ $\sF$ are\/ $\kappa$\+directed colimits
(in fact, $\kappa$\+directed unions) of $\kappa$\+presentable modules
from\/~$\sF$.
\end{prop}

\begin{proof}
 This is a direct corollary of the Hill lemma~\cite[Theorem~6]{StT},
\cite[Theorem~7.10]{GT}, \cite[Theorem~2.1]{Sto}.
 Let $F$ be a module filtered by modules from~$\sS$.
 Then the Hill lemma provides a complete lattice of submodules in $F$
such that every subset of cardinality less than~$\kappa$ in $F$
is contained in a $\kappa$\+presentable submodule of $F$ belonging to
this complete lattice, and every module belonging to the lattice is
filtered by~$\sS$.
 The family of all $\kappa$\+presentable submodules of $F$ belonging
to the lattice is $\kappa$\+directed by inclusion, and $F$ is
the directed union of these submodules.
\end{proof}

 An $R$\+module is said to be \emph{$<\kappa$\+generated} if it has
a set of generators of cardinality less than~$\kappa$.
 An associative ring $R$ is said to be \emph{right\/
$<\kappa$\+Noetherian} if every submodule of a $<\kappa$\+generated
right $R$\+module is $<\kappa$\+generated, or equivalently, every
right ideal in $R$ is $<\kappa$\+generated.
 \emph{Left\/ $<\kappa$\+Noetherian rings} are defined similarly.
 The $<\aleph_1$\+generated modules are called \emph{countably
generated}, and the right $<\aleph_1$\+Noetherian rings are called
\emph{right countably Noetherian}.

\begin{prop} \label{projective-dimension-m}
 Let $\kappa$~be an uncountable regular cardinal and $R$ be a right\/
$<\nobreak\kappa$\+Noetherian associative ring.
 Then, for every integer $m\ge0$, the class of all right $R$\+modules
of projective dimension at most~$m$ is $\kappa$\+deconstructible.
\end{prop}

\begin{proof}
 This result goes back to~\cite[Corollaire~II.3.2.5]{RG}
(for $\kappa=\aleph_1$) and~\cite[Proposition~4.1]{AEJO}.
 For all successor cardinals~$\kappa$, it is a particular case
of~\cite[Theorem~3.4]{SlT}, and the general case is similar.
\end{proof}

 For any set $X$, we denote by $|X|$ the cardinality of~$X$.
 The successor cardinal of a cardinal~$\nu$ is denoted by~$\nu^+$.
 We will use the notation $\rho=|R|+\aleph_0$ for the minimal infinite
cardinal greater than or equal to the cardinality of the ring~$R$.
 Notice that, for any given cardinal $\nu\ge\rho$, an $R$\+module is
$\nu^+$\+presentable if and only if it has cardinality at most~$\nu$.

\begin{prop} \label{enochs-slavik-trlifaj}
 Let $R$ be an associative ring; put $\rho=|R|+\aleph_0$.
 Then \par
\textup{(a)} the class of all flat $R$\+modules
is $\rho^+$\+deconstructible; \par
\textup{(b)} for every integer $m\ge0$, the class of all $R$\+modules
of flat dimension at most~$m$ is $\rho^+$\+deconstructible.
\end{prop}

\begin{proof}
 Part~(a) is~\cite[Lemma~1 and Proposition~2]{BBE}.
 Part~(b) follows from part~(a) by virtue of~\cite[Theorem~3.4]{SlT}.
\end{proof}

 Given a class of $R$\+modules $\sS\subset\Modr R$, one denotes by
$\sS^{\perp_1}\subset\Modr R$ the class of all modules $M\in\Modr R$
such that $\Ext^1_R(S,M)=0$ for all $S\in\sS$.
 The following result is known as the \emph{Eklof
lemma}~\cite[Lemma~1]{ET}, \cite[Lemma~6.2]{GT}.

\begin{lem}
 For any class of modules\/ $\sS\subset\Modr R$ one has\/
$\sS^{\perp_1}=\Fil(\sS)^{\perp_1}$.  \qed
\end{lem}

 Given two cardinals $\nu$ and~$\lambda$, one denotes by
$\nu^{<\lambda}$ the supremum of the cardinals $\nu^\mu$
taken over all the cardinals $\mu<\lambda$.
 The following lemma is a generalization of~\cite[Lemma~10.5]{GT};
it can be found in~\cite[Lemma~4.1]{CS}.

\begin{lem} \label{systems-of-equations}
 Let $R$ be a ring; put $\rho=|R|+\aleph_0$.
 Let $M$ be an $R$\+module, $\lambda$~be a regular cardinal, and
$\nu$~be a cardinal such that $\rho\le\nu$ and $\nu^{<\lambda}=\nu$.
 Then for every subset $X\subset M$ of cardinality at most~$\nu$
there exists a submodule $N\subset M$ of cardinality at most~$\nu$
such that $X\subset N$ and the following property holds:
 Every system of less than~$\lambda$ nonhomogeneous $R$\+linear
equations in less than~$\lambda$ variables, with parameters from $N$,
has a solution in $N$ provided that it has a solution in~$M$.
\end{lem}

\begin{proof}
 The argument is similar to the one in~\cite[Lemma~10.5]{GT}.
 Notice that one has $\lambda\le\nu$.
 The submodule $N\subset M$ is constructed as the union of an increasing
chain of submodules $(N_i\subset M)_{0\le i<\lambda}$, with
the cardinality of $N_i$ not exceeding~$\nu$ for every~$i$.
 Let $N_0$ be the submodule spanned by $X$ in~$M$.

 For a successor ordinal $j=i+1<\lambda$, we define $N_{i+1}$ by
adjoining to $N$ one solution of every system of less than~$\lambda$
nonhomogeneous $R$\+linear equations in less than~$\lambda$ variables
with parameters from $N_i$ that has a solution in~$M$.
 As such a system of equations has less than~$\lambda$ coefficients
from~$R$ and less than~$\lambda$ parameters from~$N_i$, the cardinality
of the set of all such systems of equations is not greater than~$\nu$.
 For a limit ordinal $j<\lambda$, we put $N_j=\bigcup_{i<j}N_i$.

 Now if a system of less than~$\lambda$ nonhomogeneous $R$\+linear
equations in less than~$\lambda$ variables has parameters in
$N=\bigcup_{i<\lambda}N_i$, then all these parameters belong to
$N_i$ for some $i<\lambda$.
 Hence such system of equations has a solution in~$N_{i+1}$.
\end{proof}

 An $R$\+module $S$ is said to be \emph{FP$_2$} \,\cite[Section~5.2]{GT}
if there exists an exact sequence $P_2\rarrow P_1\rarrow P_0\rarrow S
\rarrow0$ with finitely generated projective $R$\+modules $P_0$, $P_1$,
and~$P_2$.
 So, over a right coherent ring, all finitely presentable right modules
are FP$_2$.
 Similarly, let us say that an $R$\+module $S$ is
\emph{$\lambda$\+P$_2$} if there exists an exact sequence
$P_2\rarrow P_1\rarrow P_0\rarrow S\rarrow0$, where $P_0$, $P_1$, and
$P_2$ are projective $R$\+modules with less than~$\lambda$ generators.

\begin{prop} \label{right-Ext-1-orthogonal-accessible}
 Let $R$ be a ring; put $\rho=|R|+\aleph_0$.
 Let $\lambda$~be a regular cardinal, $\nu$~be a cardinal such that
$\rho\le\nu$ and $\nu^{<\lambda}=\nu$, and\/ $\sS$ be a set of
$\lambda$\+P$_2$ right $R$\+modules.
 Then the full subcategory\/ $\sS^{\perp_1}\subset\Modr R$ is closed
under $\lambda$\+directed colimits, and every module $M\in\sS^{\perp_1}$
is a $\nu^+$\+directed union of the $\nu^+$\+directed poset of all
the $\nu^+$\+presentable submodules of $M$ belonging
to\/~$\sS^{\perp_1}$.
 Consequently, the category\/ $\sS^{\perp_1}$ is $\nu^+$\+accessible
with directed colimits of $\lambda$\+indexed chains, and
the $\nu^+$\+presentable objects of\/ $\sS^{\perp_1}$ are precisely
all the $\nu^+$\+presentable right $R$\+modules belonging
to\/~$\sS^{\perp_1}$.
\end{prop}

\begin{proof}
 One can easily see that for any $\lambda$\+P$_2$ module $S$
there exists an exact sequence $P_2\rarrow P_1\rarrow P_0\rarrow S
\rarrow0$ such that $P_0$, $P_1$, and $P_2$ are free modules with
less than $\lambda$~generators.
 Then, for any $R$\+module $M$, the group $\Ext^1_R(S,M)$ can be
computed as the middle cohomology group of the complex
$\Hom_R(P_0,M)\rarrow\Hom_R(P_1,M)\rarrow\Hom_R(P_2,M)$.
 This proves the first assertion of the proposition.
 Furthermore, the property of an $R$\+module $M$ to belong to the class
$\{S\}^{\perp_1}\subset\Modr R$ is expressed by solvability of certain
systems of less than~$\lambda$ nonhomogeneous $R$\+linear equations
in less than~$\lambda$ variables with parameters in~$M$.
 Hence the second assertion of the proposition follows from
Lemma~\ref{systems-of-equations}.
 Notice that an $R$\+module is $\nu^+$\+presentable if and only if
its cardinality does not exceed~$\nu$.
\end{proof}

\begin{cor} \label{injectives-accessible}
 Let $\lambda$ be a regular cardinal and $R$ be a right\/
$<\lambda$\+Noetherian ring.
 Let $\nu$~be a cardinal such that $\rho=|R|+\aleph_0\le\nu$
and $\nu^{<\lambda}=\nu$.
 Then the full subcategory of injective $R$\+modules is closed
under $\lambda$\+directed colimits in $\Modr R$.
 The category of injective right $R$\+modules is $\nu^+$\+accessible
with directed colimits of $\lambda$\+indexed chains, and
the $\nu^+$\+presentable objects of this category are precisely
all the injective $R$\+modules that are $\nu^+$\+presentable
in $\Modr R$\, (i.~e., injective $R$\+modules of cardinality
at most~$\nu$).
 Moreover, every injective right $R$\+module is a $\nu^+$\+directed
union of $\nu^+$\+presentable injective $R$\+modules. 
\end{cor}

\begin{proof}
 Take $\sS$ to be the set of all cyclic right $R$\+modules $R/I$
(where $I$ ranges over all the right ideals in~$R$), and
apply Proposition~\ref{right-Ext-1-orthogonal-accessible}.
\end{proof}

\Section{Two-Sided Resolutions by Accessible Classes}
\label{two-sided-resolutions-secn}

 In this section we begin our discussion of accessibility of
categories of (co)resolutions based on the techniques described
in Section~\ref{accessible-and-acyclic-secn}.
 We start with a general abstract formulation before passing to
a finite/countable special case.

\begin{prop} \label{two-sided-by-accessible-classes-prop}
 Let $\kappa$~be a regular cardinal and\/ $\lambda<\kappa$ be a smaller
infinite cardinal.
 Let $R$ be a right\/ $<\kappa$\+coherent ring, and let
$(\sT_n)_{n\in\boZ}$ be a sequence of classes of $\kappa$\+presentable
right $R$\+modules, $\sT_n\subset(\Modr R)_{<\kappa}$.
 Assume that, for every $n\in\boZ$, the class of $R$\+modules\/ $\sT_n$
is closed under direct summands and the class of $R$\+modules
$\varinjlim_{(\kappa)}\sT_n$ is closed under colimits of
$\lambda$\+indexed chains in\/ $\Modr R$.
 Then the category\/ $\sC$ of all acyclic complexes of right $R$\+modules
$C^\bu$ with $C^n\in\varinjlim_{(\kappa)}\sT_n$ for every $n\in\boZ$
is $\kappa$\+accessible.
 The $\kappa$\+presentable objects of the category\/ $\sC$ are all
the acyclic complexes of $R$\+modules $T^\bu$ with $T^n\in\sT_n$ for
every $n\in\boZ$.
 Consequenly, every acyclic complex of $R$\+modules $C^\bu$ with
the terms $C^n\in\varinjlim_{(\kappa)}\sT_n$ is a $\kappa$\+directed
colimit of acyclic complexes $T^\bu$ with the terms $T^n\in\sT_n$.
\end{prop}

\begin{proof}
 This is an application of Theorem~\ref{pseudopullback-theorem}
together with Propositions~\ref{accessible-subcategory},
\ref{product-proposition}, and~\ref{acyclic-complexes-prop}.
 Denote by $\sK_1$ the category of acyclic complexes of right
$R$\+modules from Proposition~\ref{acyclic-complexes-prop}, and
let $\sK_2=\prod_{n\in\boZ}(\varinjlim_{(\kappa)}\sT_n)$ be
the Cartesian product of the full subcategories
$\varinjlim_{(\kappa)}\sT_n\subset\Modr R$, taken over all
the integers $n\in\boZ$.

 Let $\sL=\prod_{n\in\boZ}\Modr R$ be the Cartesian product of
$\boZ$~copies of the abelian category of right $R$\+modules.
 Consider the following functors $\Phi_1\:\sK_1\rarrow\sL$ and
$\Phi_2\:\sK_2\rarrow\sL$.
 The functor $\Phi_1$ takes an acyclic complex of right $R$\+modules
$A^\bu$ to the collection of modules $(A^n)_{n\in\boZ}$.
 The functor $\Phi_2$ is the Cartesian product of the identity
inclusion functors $\varinjlim_{(\kappa)}\sT_n\rarrow\Modr R$,
taken over all $n\in\boZ$.
 Then the pseudopullback of the pair of functors $\Phi_1$ and $\Phi_2$
is equivalent to the desired category~$\sC$.

 The category $\sK_1$ is $\kappa$\+accessible by
Proposition~\ref{acyclic-complexes-prop}.
 The categories $\sK_2$ and $\sL$ are $\kappa$\+accessible by
Propositions~\ref{accessible-subcategory} and~\ref{product-proposition}.
 Theorem~\ref{pseudopullback-theorem} is applicable; it tells
that the category $\sC$ is $\kappa$\+accessible and provides
the desired description of the full subcategory of
$\kappa$\+presentable objects in~$\sC$.
\end{proof}

\begin{thm} \label{two-sided-by-lim-classes-theorem}
 Let $R$ be a right countably coherent ring, and let
$(\sS_n)_{n\in\boZ}$ be a sequence of sets of finitely presentable
right $R$\+modules, $\sS_n\subset(\Modr R)_{<\aleph_0}$.
 Then the category\/ $\sC$ of all acyclic complexes of right
$R$\+modules $C^\bu$ with $C^n\in\varinjlim\sS_n$ for every
$n\in\boZ$ is\/ $\aleph_1$\+accessible.
 The\/ $\aleph_1$\+presentable objects of the category\/ $\sC$ are
all the acyclic complexes $T^\bu$ such that $T^n\in\varinjlim\sS_n$
and $T^n$ is a countably presentable $R$\+module for every $n\in\boZ$.
 Consequently, every acyclic complex of $R$\+modules $C^\bu$ with
the terms $C^n\in\varinjlim\sS_n$ is an $\aleph_1$\+directed colimit
of acyclic complexes of countably presentable $R$\+modules $T^\bu$
with the terms $T^n\in\varinjlim\sS_n$.
\end{thm}

\begin{proof}
 This is a particular case of
Proposition~\ref{two-sided-by-accessible-classes-prop}.
 Put $\lambda=\aleph_0$, \,$\kappa=\aleph_1$, and denote by
$\sT_n\subset\Modr R$ the class of all direct summands of countable
directed colimits of modules from $\sS_n$, for every $n\in\boZ$.
 Then $\varinjlim\sS_n=\varinjlim_{(\aleph_1)}\sT_n$, as per the proof
of Lemma~\ref{kappa-presentable-directed-colimit-lemma}.
 The class $\varinjlim\sS_n$ is closed under directed colimits in
$\Modr R$ by Proposition~\ref{accessible-subcategory}.
\end{proof}

 Given an acyclic complex of $R$\+modules $C^\bu$, we denote by
$Z^0(C^\bu)$ the module of degree~$0$ cocycles in~$C^\bu$.
 So $Z^0(C^\bu)$ is the kernel of the differential
$d^0\:C^0\rarrow C^1$, or equivalently, the cokernel of
the differential $d^{-2}\:C^{-2}\rarrow C^{-1}$.

\begin{cor} \label{two-sided-by-lim-classes-cor}
 Let $R$ be a right countably coherent ring, and let
$(\sS_n)_{n\in\boZ}$ be a sequence of sets of finitely presentable
right $R$\+modules, $\sS_n\subset(\Modr R)_{<\aleph_0}$.
 Denote by\/ $\sM$ the class of all right $R$\+modules $M$ of
the form $M=Z^0(C^\bu)$, where $C^\bu$ is an acyclic complex
of $R$\+modules with the terms $C^n\in\varinjlim\sS_n$.
 Then every module from\/ $\sM$ is an\/ $\aleph_1$\+directed colimit
of countably presentable modules from\/~$\sM$.
\end{cor}

\begin{proof}
 Let us emphasize that there is \emph{no} claim about closedness of
the class $\sM$ under $\aleph_1$\+directed colimits in this corollary.
 The assertion of the corollary follows directly from the last
assertion of Theorem~\ref{two-sided-by-lim-classes-theorem}.
 One needs to observe that, in any acyclic complex of countably
presentable modules, the modules of cocycles are also countably
presentable.
\end{proof}

\begin{cor} \label{two-sided-by-lim-classes-direct-summand}
 In the context of Corollary~\ref{two-sided-by-lim-classes-cor},
let $M$ be a countably presentable $R$\+module belonging to
the class\/~$\sM$.
 Then $M$ is a direct summand of an $R$\+module $N=Z^0(T^\bu)$,
where $T^\bu$ is an acyclic complex of countably presentable
$R$\+modules with the terms $T^n\in\varinjlim\sS_n$.
\end{cor}

\begin{proof}
 It is clear from Theorem~\ref{two-sided-by-lim-classes-theorem} that
every $R$\+module $M\in\sM$ is the colimit of an $\aleph_1$\+directed
diagram of $R$\+modules $(N_\xi)_{\xi\in\Xi}$ such that $N_\xi=
Z^0(T^\bu_\xi)$ for some acyclic complexes of countably presentable
modules $T^\bu_\xi$ with the terms $T^n_\xi\in\varinjlim\sS_n$.
 Now if $M$ is countably presentable, then it follows that there exists
$\xi\in\Xi$ such that $M$ is a direct summand of $N_\xi$.
\end{proof}

\Section{Modules of Finite Flat Dimension}

 Recall from the introduction that, already over commutative Noetherian
local rings $R$ of Krull dimension~$2$, a module of flat dimension~$1$
need not be a directed colimit of modules of projective dimension
at most~$1$ \,\cite[Example~8.5]{BH}, \cite[Theorem~B]{HG}.
 Any finitely generated $R$\+module of flat dimension~$\le1$ would, of
course, have projective dimension~$\le1$ (since finitely presentable
flat modules are always projective).
 So a module of flat dimension~$1$ \emph{need not} be a directed colimit
of finitely generated flat modules of flat dimension at most~$1$.

 By contrast, in this section we show that any module of flat
dimension~$m$ is a directed colimit of countably presentable modules
of flat dimension~$\le m$.

\begin{thm} \label{finite-flat-resolutions-theorem}
 Let $R$ be a right countably coherent ring and $m\ge0$ be an integer.
 Then the category of all exact sequences of flat right $R$\+modules\/
$0\rarrow F_m\rarrow F_{m-1}\rarrow\dotsb\rarrow F_0$ is\/
$\aleph_1$\+accessible.
 The\/ $\aleph_1$\+presentable objects of this category are precisely
all the exact sequences\/ $0\rarrow T_m\rarrow T_{m-1}\rarrow\dotsb
\rarrow T_0$ with countably presentable flat right $R$\+modules $T_n$,
\,$0\le n\le m$.
 Consequently, every finite flat resolution\/ $0\rarrow F_m\rarrow
F_{m-1}\rarrow\dotsb\rarrow F_0$ in\/ $\Modr R$ is
an\/ $\aleph_1$\+directed colimit of finite flat resolutions\/
$0\rarrow T_m\rarrow T_{m-1}\rarrow\dotsb\rarrow T_0$ of the same
length~$m$ with countably presentable flat modules~$T_n$,
\,$0\le n\le m$.
\end{thm}

\begin{proof}
 This is a particular case of
Theorem~\ref{two-sided-by-lim-classes-theorem}.
 Take $\sS_n$ to be the set of all finitely generated projective
(or free) right $R$\+modules for all $-m-1\le n\le-1$, \ $\sS_0$ to be
the set of all finitely presentable right $R$\+modules, and
$\sS_n=\{0\}$ for all $n\le-m-2$ and all $n\ge1$.
 Then $\varinjlim\sS_n$ is the class of all flat right $R$\+modules
for all $-m-1\le n\le-1$, \ $\varinjlim\sS_0=\Modr R$, and
$\varinjlim\sS_n=\{0\}$ for all $n\le-m-2$ and all $n\ge1$.
 So the category $\sC$ from
Theorem~\ref{two-sided-by-lim-classes-theorem} is equivalent to
the category of finite flat resolutions we are interested in.
\end{proof}

\begin{cor} \label{finite-flat-dimension-cor}
 Let $R$ be a right countably coherent ring and $m\ge0$ be an integer.
 Denote by\/ $\sF_m\subset\Modr R$ the full subcategory of all right
$R$\+modules of flat dimension at most~$m$.
 Then the category\/ $\sF_m$ is\/ $\aleph_1$\+accessible.
 The\/ $\aleph_1$\+presentable objects of\/ $\sF_m$ are precisely all
the countably presentable right $R$\+modules of flat dimension\/~$\le m$
(i.~e., those modules from\/ $\sF_m$ that are countably presentable in\/
$\Modr R$).
 So every right $R$\+module of flat dimension~$m$ is
an\/ $\aleph_1$\+directed colimit of countably presentable $R$\+modules
of flat dimension at most~$m$.
\end{cor}

\begin{proof}
 It is easy to see that the full subcategory $\sF_m$ is closed under
directed colimits in $\Modr R$ (since the functor Tor preserves
directed colimits).
 In view of Proposition~\ref{accessible-subcategory}, it suffices to
check the last assertion of the corollary, which follows immediately
from Theorem~\ref{finite-flat-resolutions-theorem}.
 Notice that the cokernel of any morphism of countably presentable
modules is countably presentable.
\end{proof}

\begin{cor}
 Let $R$ be a right countably coherent ring and $m\ge0$ be an integer.
 Then every right $R$\+module of flat dimension~$m$ is
an\/ $\aleph_1$\+directed colimit of countably presentable right
$R$\+modules of projective dimension at most~$m+1$.
\end{cor}

\begin{proof}
 This is a corollary of Corollary~\ref{finite-flat-dimension-cor}.
 The point is that any countably presentable right $R$\+module of
flat dimension~$\le m$ has projective dimension~$\le m+1$, since any
countably presentable flat module has projective dimension at most~$1$
(by~\cite[Corollary~2.23]{GT}
or Corollary~\ref{aleph-m-presentable-flat-module-cor} above).
\end{proof}

 Notice that Corollary~\ref{finite-flat-dimension-cor} (or more
generally, Proposition~\ref{two-sided-by-accessible-classes-prop})
gives a better cardinality estimate for the accessibility rank of
the category of $R$\+modules of flat dimension~$\le m$ than
Proposition~\ref{deconstructible-as-directed-colimit} combined with
Proposition~\ref{enochs-slavik-trlifaj}(b).

\Section{Flatly Coresolved Modules}
\label{flatly-coresolved-secn}

 By a \emph{flat coresolution} we mean an exact sequence $F^0\rarrow
F^1\rarrow F^2\rarrow\dotsb$, where $F^n$ are flat modules for all
$n\ge0$.
 An $R$\+module $M$ is said to be \emph{flatly coresolved} if
there exists an exact sequence of $R$\+modules $0\rarrow M\rarrow
F^0\rarrow F^1\rarrow F^2\rarrow\dotsb$ with flat $R$\+modules~$F^n$.

\begin{thm} \label{flat-coresolutions-theorem}
 Let $R$ be a right countably coherent ring.
 Then the category of flat coresolutions $F^0\rarrow F^1\rarrow F^2
\rarrow\dotsb$ in $\Modr R$ is\/ $\aleph_1$\+accessible.
 The\/ $\aleph_1$\+presentable objects of this category are precisely
all the flat coresolutions $T^0\rarrow T^1\rarrow T^2\rarrow\dotsb$
with countably presentable flat right $R$\+modules $T^n$, \,$n\ge0$.
 Consequently, every flat coresolution $F^0\rarrow F^1\rarrow F^2
\rarrow\dotsb$ in $\Modr R$ is an\/ $\aleph_1$\+directed colimit of
flat coresolutions $T^0\rarrow T^1\rarrow T^2\rarrow\dotsb$
with countably presentable flat $R$\+modules $T^n$, \,$n\ge0$.
\end{thm}

\begin{proof}
 This is another particular case of
Theorem~\ref{two-sided-by-lim-classes-theorem}.
 Take $\sS_n$ to be the set of all finitely generated projective
(or free) right $R$\+modules for all $n\ge0$, \ $\sS_{-1}$ to be
the set of all finitely presentable right $R$\+modules, and
$\sS_n=\{0\}$ for all $n\le-2$.
 Then $\varinjlim\sS_n$ is the class of all flat right $R$\+modules for
all $n\ge0$, \ $\varinjlim\sS_{-1}=\Modr R$, and $\varinjlim\sS_n=0$
for all $n\le-2$.
 So the category $\sC$ from
Theorem~\ref{two-sided-by-lim-classes-theorem} is equivalent to
the category of flat coresolutions we are interested in.
\end{proof}

\begin{cor} \label{flatly-coresolved-modules-cor}
 Let $R$ be a right countably coherent ring.
 Then any flatly coresolved right $R$\+module is
an\/ $\aleph_1$\+directed colimit of countably presentable
flatly coresolved right $R$\+modules.
\end{cor}

\begin{proof}
 This is a corollary of Theorem~\ref{flat-coresolutions-theorem} and
a particular case of Corollary~\ref{two-sided-by-lim-classes-cor}.
 Let us emphasize once again that there is \emph{no} claim about
the class of flatly coresolved modules being closed under
$\aleph_1$\+directed colimits in this corollary.
 To deduce the assertion of the corollary from
Theorem~\ref{flat-coresolutions-theorem}, one needs to observe
that the kernel of any morphism of countably presentable right modules
over a right countably coherent ring $R$ is countably presentable.
\end{proof}

\begin{cor} \label{flat-coresolutions-direct-summands}
 Let $R$ be a right countably coherent ring and $M$ be a countably
presentable flatly coresolved right $R$\+module.
 Then $M$ is a direct summand of an $R$\+module $N$ admitting a flat
coresolution\/ $0\rarrow N\rarrow T^0\rarrow T^1\rarrow T^2\rarrow
\dotsb$ with countably presentable flat right $R$\+modules~$T^n$,
\,$n\ge0$.
\end{cor}

\begin{proof}
 This is a corollary of Theorem~\ref{flat-coresolutions-theorem}
and a particular case of
Corollary~\ref{two-sided-by-lim-classes-direct-summand}.
\end{proof}

\Section{Dualizing Complexes and F-Totally Acyclic Complexes}
\label{dualizing-and-F-totally-acyclic-secn}

 An acyclic complex of flat right $R$\+modules $F^\bu$ is said to be
\emph{F\+totally acyclic}~\cite[Section~2]{EFI} if the complex of
abelian groups $F^\bu\ot_RJ$ is acyclic for every injective left
$R$\+module~$J$.
 A right $R$\+module $M$ is said to be \emph{Gorenstein-flat} if
there exists an F\+totally acyclic complex of flat right $R$\+modules
$F^\bu$ such that $M\simeq Z^0(F^\bu)$ is its module of cocycles.

 The aim of the following four sections,
Sections~\ref{dualizing-and-F-totally-acyclic-secn}\+-%
\ref{gorenstein-flats-secn}, is to establish some accessibility
properties of the categories of F\+totally acyclic complexes and
Gorenstein-flat modules, under suitable assumptions.
 In particular, the present section is purported to prepare
ground for the next one.

 We recall that a right $R$\+module $K$ is said to be
\emph{fp\+injective} (or ``absolutely pure'') if $\Ext^1_R(T,K)=0$
for all finitely presentable right $R$\+modules~$T$.
 The definition of an fp\+injective left module is similar.

 Given an abelian category $\sA$, we denote by $\sD^\bb(\sA)$
the bounded derived category of~$\sA$.
 A bounded complex of left $R$\+modules $K^\bu$ is said to have
\emph{finite fp\+injective dimension} if it is isomorphic, as an object
of the derived category of left $R$\+modules $\sD^\bb(R\Modl)$,
to a bounded complex of fp\+injective left $R$\+modules.

 Let $R$ and $S$ be two associative rings.
 Assume that the ring $S$ is right coherent.
 In the following definition we list some properties of a complex
of $R$\+$S$\+bimodules $D^\bu$ that are relevant for the purposes
of this section.
 For comparison, see the definition of a \emph{dualizing complex}
in~\cite[Section~4]{Pfp}.

 We will say that a bounded complex of $R$\+$S$\+bimodules $D^\bu$ is
a \emph{right dualizing complex} for the rings $R$ and $S$ if it
satisfies the following conditions:
\begin{enumerate}
\renewcommand{\theenumi}{\roman{enumi}}
\item the terms of the complex $D^\bu$ are fp\+injective as right
$S$\+modules, and the whole complex $D^\bu$ has finite fp\+injective
dimension as a complex of left $R$\+modules;
\item the $R$\+$S$\+bimodules of cohomology of the complex $D^\bu$
are finitely presentable as right $S$\+modules;
\item the homothety map $R\rarrow
\Hom_{\sD^\bb(\Modr S)}(D^\bu,D^\bu[*])$ is an isomorphism of
graded rings.
\end{enumerate}

 The following proposition is the main result of this section.
 It is our version of~\cite[Lemma~1.7]{Jor}.

\begin{prop} \label{F-totally-acyclics-described-by-dualizing}
 Let $R$ be a ring, $S$ be a right coherent ring, and $D^\bu$ be
a right dualizing complex of $R$\+$S$\+bimodules.
 Then an acyclic complex of flat right $R$\+modules $F^\bu$ is
F\+totally acyclic if and only if the complex of right $S$\+modules
$F^\bu\ot_RD^\bu$ is acyclic.
\end{prop}

 The proof of
Proposition~\ref{F-totally-acyclics-described-by-dualizing}
is based on a sequence of lemmas.

\begin{lem} \label{termwise-flat-tensored-with-bounded-acyclic}
 Let $F^\bu$ be a complex of flat right $R$\+modules and $L^\bu$
be a bounded acyclic complex of left $R$\+modules.
 Then the complex of abelian groups $F^\bu\ot_R L^\bu$ is acyclic.
\end{lem}

\begin{proof}
 The point is that the complex of abelian groups $F^n\ot_R L^\bu$
is acyclic for every $n\in\boZ$.
 The total complex of any bounded acyclic complex of complexes of
abelian groups is acyclic.
\end{proof}

\begin{lem} \label{tensor-with-fp-injective-module}
 Let $F^\bu$ be a complex of right $R$\+modules such that the complex
of abelian groups $F^\bu\ot_RJ$ is acyclic for every injective left
$R$\+module~$J$.
 Then the complex of abelian groups $F^\bu\ot_RK$ is acyclic for
every fp\+injective left $R$\+module~$K$.
\end{lem}

\begin{proof}
 This is~\cite[proof of Lemma~2.8\,(1)\,$\Rightarrow$\,(2)]{MD}.
 Let an fp\+injective $R$\+module $K$ be a submodule in an injective
$R$\+module~$J$; then $K$ is a pure submodule in~$J$ (see, e.~g.,
\cite[Definition~2.6 and Lemma~2.19]{GT}).
 The point is that, whenever some left $R$\+module $M$ is a pure
submodule in a left $R$\+module $L$, and $F^\bu$ is a complex of
right $R$\+modules such that the complex of abelian groups
$F^\bu\ot_RL$ is acyclic, it follows that the complex of abelian
groups $F^\bu\ot_RM$ is acyclic as well.
 One can prove this by observing that $\Hom_\boZ(L,\boQ/\boZ)
\rarrow\Hom_\boZ(M,\boQ/\boZ)$ is a split epimorphism of right
$R$\+modules; hence the complex $\Hom_\boZ(F^\bu\ot_RM,\>\boQ/\boZ)$ is
a direct summand of the complex $\Hom_\boZ(F^\bu\ot_RL,\>\boQ/\boZ)$.
\end{proof}

\begin{cor} \label{tensor-with-bounded-fp-injective-dimension}
 Let $F^\bu$ be an F\+totally acyclic complex of flat right
$R$\+modules and $K^\bu$ be a bounded complex of left $R$\+modules.
 Assume that the complex $K^\bu$ has finite fp\+injective dimension.
 Then the complex of abelian groups $F^\bu\ot_RK^\bu$ is acyclic.
\end{cor}

\begin{proof}
 Combine Lemmas~\ref{termwise-flat-tensored-with-bounded-acyclic}
and~\ref{tensor-with-fp-injective-module}.
\end{proof}

\begin{lem} \label{acyclic-tensored-with-bounded-flat}
 Let $S$ be a ring, $H^\bu$ be an acyclic complex of right $S$\+modules,
and $G^\bu$ be a bounded complex of flat left $S$\+modules.
 Then the complex of abelian groups $H^\bu\ot_SG^\bu$ is acyclic.
\end{lem}

\begin{proof}
 The point is that the total complex of any bounded complex of
acyclic complexes of abelian groups is acyclic.
\end{proof}

\begin{lem} \label{Hom-module-flat}
 Let $R$ be a ring, $S$ be a right coherent ring, $E$ be
an $R$\+$S$\+bimodule that is fp\+injective as a right $S$\+module,
and $J$ be an injective left $R$\+module.
 Then the left $S$\+module\/ $\Hom_R(E,J)$ is flat.
\end{lem}

\begin{proof}
 This is~\cite[Lemma~4.1(b)]{Pfp}.
 The point is that the functor $N\longmapsto N\ot_S\Hom_R(E,J)
\simeq\Hom_R(\Hom_S(N,E),J)$ is exact on the abelian category of
finitely presentable right $S$\+modules~$N$.
\end{proof}

\begin{lem} \label{evaluation-quasi-isomorphism}
 Let $R$ be a ring, $S$ be a right coherent ring, and $D^\bu$ be
a right dualizing complex of $R$\+$S$\+bimodules.
 Then the natural (evaluation) morphism of complexes of left
$R$\+modules $D^\bu\ot_S\Hom_R(D^\bu,J)\rarrow J$ is a quasi-isomorphism
for every injective left $R$\+module~$J$.
\end{lem}

\begin{proof}
 This lemma only uses conditions~(ii\+-iii) and the first part of
condition~(i) from the definition of a right dualizing complex above.
 See~\cite[Lemma~4.2(b)]{Pfp}.
\end{proof}

\begin{proof}[Proof of
Proposition~\ref{F-totally-acyclics-described-by-dualizing}]
 The ``only if'' implication is provided by
Corollary~\ref{tensor-with-bounded-fp-injective-dimension}.
 To prove the ``if'', let $J$ be an injective left $R$\+module.
 By Lemma~\ref{evaluation-quasi-isomorphism}, we have
a quasi-isomorphism of bounded complexes of left $R$\+modules
$D^\bu\ot_S\Hom_R(D^\bu,J)\rarrow J$.
 In view of Lemma~\ref{termwise-flat-tensored-with-bounded-acyclic},
in order to show that the complex $F^\bu\ot_RJ$ is acyclic, it
suffices to check that so is the complex
$F^\bu\ot_RD^\bu\ot_S\Hom_R(D^\bu,J)$.
 Now, by assumption, the complex of right $S$\+modules $F^\bu\ot_RD^\bu$
is acyclic.
 By Lemma~\ref{Hom-module-flat}, \,$\Hom_R(D^\bu,J)$ is a bounded
complex of flat left $S$\+modules. 
 It remains to refer to Lemma~\ref{acyclic-tensored-with-bounded-flat}.
\end{proof}

\begin{rem} \label{gorenstein-flats-in-gorenstein-case}
 The concepts of Gorenstein homological algebra become less complicated
over a ring which is itself Gorenstein in a suitable sense.
 Notice that, for any acyclic complex of flat right $R$\+modules $F^\bu$
and any left $R$\+module $G$ of finite flat dimension, the complex of
abelian groups $F^\bu\ot_RG$ is acyclic.
 This follows from
Lemmas~\ref{termwise-flat-tensored-with-bounded-acyclic}
and~\ref{acyclic-tensored-with-bounded-flat}.
 Now if every injective left $R$\+module has finite flat dimension,
then every acyclic complex of flat right $R$\+modules is F\+totally
acyclic, and every flatly coresolved right $R$\+module is
Gorenstein-flat.
 Assuming additionally that the ring $R$ is right countably coherent,
the results of Theorem~\ref{two-sided-by-lim-classes-theorem},
Theorem~\ref{flat-coresolutions-theorem},
Corollary~\ref{flatly-coresolved-modules-cor},
and Corollary~\ref{flat-coresolutions-direct-summands} become directly
applicable (essentially) as properties of F\+totally acyclic complexes
and Gorenstein-flat modules.
\end{rem}

\Section{F-Totally Acyclic Complexes as Directed Colimits}
\label{F-totally-acyclics-as-dir-colims-secn}

 We start with a very general result of \v Saroch and
\v St\!'ov\'\i\v cek.

\begin{prop} \label{saroch-stovicek-gorenstein-flat-basics}
 Let $R$ be an associative ring.  Then \par
\textup{(a)} the class of Gorenstein-flat modules is closed under
directed colimits in\/ $\Modr R$; \par
\textup{(b)} the class of Gorenstein-flat $R$\+modules is
$\rho^+$\+deconstructible, where $\rho=|R|+\aleph_0$.
\end{prop}

\begin{proof}
 These assertions are parts of~\cite[Corollary~4.12]{SarSt}.
\end{proof}

 We will need to assume another condition in addition to
the conditions~(i\+-iii) in the definition of a right dualizing complex
in Section~\ref{dualizing-and-F-totally-acyclic-secn}.
 We will say that a right dualizing complex of $R$\+$S$\+bimodules
$D^\bu$ has \emph{right countable type} if
\begin{enumerate}
\renewcommand{\theenumi}{\roman{enumi}}
\setcounter{enumi}{3}
\item there exist a right countably coherent ring $S'$ and a bounded
complex of $R$\+$S'$\+bimodules $\widetilde D^\bu$ whose terms are
countably presentable as right $S'$\+modules such that the complexes
$D^\bu$ and $\widetilde D^\bu$ are isomorphic as objects of the bounded 
derived category of left $R$\+modules $\sD^\bb(R\Modl)$.
\end{enumerate}
 
\begin{lem} \label{tensor-product-of-presentable}
 Let $R$ and $S$ be two rings, $T$ be a right $R$\+module, and
$E$ be an $R$\+$S$\+bimodule.
 Assume that the right $R$\+module $T$ is $\kappa$\+presentable and
the right $S$\+module $E$ is $\kappa$\+presentable, for a given
regular cardinal~$\kappa$.
 Then the right $S$\+module $T\ot_RE$ is also $\kappa$\+presentable.
\end{lem}

\begin{proof}
 Represent $T$ as the cokernel of a morphism of free $R$\+modules
with less than~$\kappa$ generators, and use the fact that the class
of all $\kappa$\+presentable $S$\+modules is closed under colimits
indexed by categories with less than~$\kappa$
morphisms~\cite[Proposition~1.16]{AR}.
\end{proof}

\begin{thm} \label{dualizing-complex-F-totally-acyclics-theorem}
 Let $R$ be a right countably coherent ring, $S$ be a right coherent
ring, and $D^\bu$ be a right dualizing complex of $R$\+$S$\+bimodules.
 Assume that $D^\bu$ is a right dualizing complex of right
countable type.
 Then the category of F\+totally acyclic complexes of flat right
$R$\+modules is\/ $\aleph_1$\+accessible.
 The\/ $\aleph_1$\+presentable objects of this category are precisely
all the F\+totally acyclic complexes of countably presentable flat
right $R$\+modules.
 Consequently, every F\+totally acyclic complex of flat right
$R$\+modules is an\/ $\aleph_1$\+directed colimit of F\+totally
acyclic complexes of countably presentable flat right $R$\+modules.
\end{thm}

\begin{proof}
 This is an application of Theorem~\ref{pseudopullback-theorem} together
with Theorem~\ref{two-sided-by-lim-classes-theorem} and
Proposition~\ref{F-totally-acyclics-described-by-dualizing}.
 Notice first of all that the full subcategory of F\+totally acyclic
complexes of flat modules is obviously closed under directed colimits
in the ambient abelian category of all complexes of modules.

 Now let $\widetilde D^\bu$ be a bounded complex of $S'$\+countably
presentable $R$\+$S'$\+bimodules provided by the definition of
a right dualizing complex of right countable type (item~(iv)) above.
 It is clear from
Lemma~\ref{termwise-flat-tensored-with-bounded-acyclic} that, for
any complex of flat right $R$\+modules $F^\bu$, the complex of right
$S$\+modules $F^\bu\ot_RD^\bu$ is acyclic if and only if the complex
of right $S'$\+modules $F^\bu\ot_R\widetilde D^\bu$ is acyclic.

 Denote by $\sK_1$ the category of acyclic complexes of flat right
$R$\+modules.
 Applying Theorem~\ref{two-sided-by-lim-classes-theorem} for
$\sS_n$ being the set of all finitely generated projective (or free)
$R$\+modules, for every $n\in\boZ$, we see that $\sK_1$ is
an $\aleph_1$\+accessible category, and obtain a description of its
$\aleph_1$\+presentable objects (cf.\
Theorem~\ref{flat-coresolutions-theorem}).

 Furthermore, denote by $\sK_2$ the category of acyclic complexes of
right $S'$\+modules.
 Proposition~\ref{acyclic-complexes-prop} tells that $\sK_2$ is
an $\aleph_1$\+accessible category and provides a description of its
full subcategory of $\aleph_1$\+presentable objects.
 Finally, let $\sL$ be the abelian category of arbitrary complexes
of right $S'$\+modules.
 Lemma~\ref{abelian-complexes-lemma}(a,c) tells that $\sL$ is
a locally $\aleph_1$\+presentable category and describes its full
subcategory of $\aleph_1$\+presentable objects.

 Let the functor $\Phi_1\:\sK_1\rarrow\sL$ take any acyclic complex of
flat right $R$\+modules $F^\bu$ to the complex of right $S'$\+modules
$F^\bu\ot_R\widetilde D^\bu$.
 Let $\Phi_2\:\sK_2\rarrow\sL$ be the identity inclusion of
the category of acyclic complexes of right $S'$\+modules into
the category of all complexes of right $S'$\+modules.
 Then the pseudopullback $\sC$ of the two functors $\Phi_1$ and
$\Phi_2$ is the category of all acyclic complexes of flat right
$R$\+modules $F^\bu$ for which the complex of right $S'$\+modules
$F^\bu\ot_R\widetilde D^\bu$ is acyclic.
 By Proposition~\ref{F-totally-acyclics-described-by-dualizing}, and
in view of the argument with
Lemma~\ref{termwise-flat-tensored-with-bounded-acyclic} above,
the category $\sC$ is the desired category of F\+totally acyclic
complexes of flat right $R$\+modules.

 Finally, by Lemma~\ref{tensor-product-of-presentable}, the functor
$\Phi_1$ takes $\aleph_1$\+presentable objects to
$\aleph_1$\+pre\-sentable objects.
 All the other assumptions of Theorem~\ref{pseudopullback-theorem}
(for $\kappa=\aleph_1$ and $\lambda=\aleph_0$) are clearly satisfied.
 Theorem~\ref{pseudopullback-theorem} tells that $\sC$ is
an $\aleph_1$\+accessible category, and provides the desired description
of its full subcategory of $\aleph_1$\+presentable objects.
\end{proof}

\begin{cor} \label{dualizing-complex-Gorenstein-flats-cor}
 Let $R$ be a right countably coherent ring, $S$ be a right coherent
ring, and $D^\bu$ be a right dualizing complex of $R$\+$S$\+bimodules.
 Assume that $D^\bu$ is a right dualizing complex of right
countable type.
 Then the category of Gorenstein-flat right $R$\+modules\/ $\GF$ is\/
$\aleph_1$\+accessible.
 The\/ $\aleph_1$\+presentable objects of\/ $\GF$ are precisely
all the Gorenstein-flat right $R$\+modules that are countably
presentable in $\Modr R$.
 So every Gorenstein-flat right $R$\+module is an\/ $\aleph_1$\+directed
colimit of countably presentable Gorenstein-flat right $R$\+modules.
\end{cor}

\begin{proof}
 The full subcategory $\GF\subset\Modr R$ is closed under directed
colimits in $\Modr R$ by
Proposition~\ref{saroch-stovicek-gorenstein-flat-basics}(a).
 On the other hand, it is obvious from
Theorem~\ref{dualizing-complex-F-totally-acyclics-theorem} that
any Gorenstein-flat right $R$\+module is an $\aleph_1$\+directed colimit
of Gorenstein-flat right $R$\+modules that are countably presentable
in $\Modr R$.
 In view of Proposition~\ref{accessible-subcategory}, all the assertions
of the corollary follow.
\end{proof}

 For a far-reaching generalization of
Corollary~\ref{dualizing-complex-Gorenstein-flats-cor}, see
Theorem~\ref{countably-coherent-Gorenstein-flats-theorem} below.

\begin{cor} \label{dualizing-complex-F-totally-acyclics-direct-summands}
 Let $R$ be a right countably coherent ring, $S$ be a right coherent
ring, and $D^\bu$ be a right dualizing complex of $R$\+$S$\+bimodules.
 Assume that $D^\bu$ is a right dualizing complex of right
countable type.
 Let $M$ be a countably presentable Gorenstein-flat right $R$\+module.
 Then $M$ is a direct summand of an $R$\+module $N$ admitting
an F\+totally acyclic two-sided resolution $T^\bu$ with countably
presentable flat right $R$\+modules $T^n$, \,$n\in\boZ$.
\end{cor}

\begin{proof}
 The argument is similar to the proof of
Corollary~\ref{two-sided-by-lim-classes-direct-summand}.
 It is clear from
Theorem~\ref{dualizing-complex-F-totally-acyclics-theorem} that every
Gorenstein-flat right $R$\+module $M$ is the colimit of
an $\aleph_1$\+directed diagram of $R$\+modules $(N_\xi)_{\xi\in\Xi}$
such that $N_\xi=Z^0(T_\xi^\bu)$ for some F\+totally acyclic complexes
of countably presentable flat right $R$\+modules~$T_\xi^\bu$.
 Now if $M$ is countably presentable, then it follows that $M$ is
a direct summand of $N_\xi$ for some $\xi\in\Xi$.
\end{proof}

\Section{Commutative Noetherian Rings with Countable Spectrum}
\label{countable-spectrum-secn}

 In this section, we prove results similar to those of
Section~\ref{F-totally-acyclics-as-dir-colims-secn}, but under
a different set of assumptions.
 Instead of existence of a dualizing complex, we assume the ring $R$
to be commutative Noetherian with small cardinality of the spectrum.

 For any $R$\+module $M$, we denote by $E_R(M)$ an injective envelope
of~$M$.
 For a prime ideal~$\p$ in a commutative ring $R$, we denote by $R_\p$
the local ring $(R\setminus\p)^{-1}R$.
 The following proposition is standard commutative algebra material.

\begin{prop} \label{matlis}
 Let $R$ be a commutative Noetherian ring.  In this setting: \par
\textup{(a)} All injective $R$\+modules are direct sums of
indecomposable injective $R$\+modules. \par
\textup{(b)} The indecomposable injective $R$\+modules, viewed up to
isomorphism, correspond bijectively to prime ideals of~$R$.
 For every prime ideal\/ $\p\in\Spec R$, the corresponding
indecomposable injective $R$\+module is the injective envelope
$E_R(R/\p)$ of the $R$\+module~$R/\p$. \par
\textup{(c)} For every prime ideal\/ $\p\in\Spec R$, the $R$\+module
$E_R(R/\p)$ is a module over the local ring $R_\p$.
 The $R_\p$\+module $E_R(R/\p)$ is (at most) countably generated.
\end{prop}

\begin{proof}
 All these results are due to Matlis~\cite{Mat}.
 Part~(a) is~\cite[Theorem~2.5]{Mat}.
 Part~(b) is~\cite[Proposition~3.1]{Mat}.
 Part~(c) is~\cite[Theorems~3.6 and~3.11]{Mat}.
\end{proof}

\begin{thm} \label{commutative-Noetherian-F-totally-acyclics-theorem}
 Let $\kappa$~be an uncountable regular cardinal and $R$ be
a commutative Noetherian ring with the cardinality of
the spectrum\/ $|\Spec R|<\kappa$.
 Then the category of F\+totally acyclic complexes of flat $R$\+modules
is $\kappa$\+accessible.
 The $\kappa$\+presentable objects of this category are precisely all
the F\+totally acyclic compexes of\/ $<\kappa$\+generated flat
$R$\+modules.
 Consequently, every F\+totally acyclic complex of flat $R$\+modules is
a $\kappa$\+directed colimit of F\+totally acyclic complexes of\/
$<\kappa$\+generated flat $R$\+modules.
\end{thm}

\begin{proof}
 This is an application of Theorem~\ref{pseudopullback-theorem}
together with Proposition~\ref{product-proposition}
and Proposition~\ref{two-sided-by-accessible-classes-prop}.
 The argument is somewhat similar to the proof of
Theorem~\ref{dualizing-complex-F-totally-acyclics-theorem}.

 Denote by $\sK_1$ the category of acyclic complexes of flat
$R$\+modules.
 Applying Proposition~\ref{two-sided-by-accessible-classes-prop}
for $\sT_n$ being the class of all $<\kappa$\+generated flat
$R$\+modules for every $n\in\boZ$ (cf.~\cite[Proposition~10.2]{Pacc})
and $\lambda=\aleph_0$, we see that $\sK_1$ is a $\kappa$\+accessible
category and obtain a description of its $\kappa$\+presentable objects.

 Furthermore, denote by $\sK_2$ the Cartesian product of the categories
of acyclic complexes of $R_\p$\+modules, taken over all the prime
ideals $\p\in\Spec R$.
 Proposition~\ref{acyclic-complexes-prop} together with
Proposition~\ref{product-proposition} tell that $\sK_2$ is
a $\kappa$\+accessible category and provide a description of its
full subcategory of $\kappa$\+presentable objects.
 Finally, denote by $\sL$ the Cartesian product of the abelian
categories of arbitrary complexes of $R_\p$\+modules, taken over
all the spectrum points $\p\in\Spec R$.
 Lemma~\ref{abelian-complexes-lemma}(a,c) together with
Proposition~\ref{product-proposition} tell that $\sL$ is
a locally $\kappa$\+presentable category and describe its full
subcategory of $\kappa$\+presentable objects.

 Let the functor $\Phi_1\:\sK_1\rarrow\sL$ take any acyclic complex
of flat $R$\+modules $F^\bu$ to the collection of all complexes
of $R_\p$\+modules $F^\bu\ot_RE_R(R/\p)$.
 Let $\Phi_2\:\sK_2\rarrow\sL$ be the Cartesian product of
the identity inclusions of the categories of acyclic complexes of
$R_\p$\+modules into the respective categories of all complexes of
$R_\p$\+modules.
 Then the pseudopullback $\sC$ of the two functors $\Phi_1$ and $\Phi_2$
is the category of all F\+totally acyclic complexes of flat
$R$\+modules (in view of Proposition~\ref{matlis}(a\+-b)).

 Finally, it follows from Proposition~\ref{matlis}(c) and
Lemma~\ref{tensor-product-of-presentable} (for $S=R_\p$) that
the functor $\Phi_1$ takes $\kappa$\+presentable objects to
$\kappa$\+presentable objects.
 All the other assumptions of Theorem~\ref{pseudopullback-theorem}
(for the given cardinal~$\kappa$ and $\lambda=\aleph_0$) are clearly
satisfied.
 Theorem~\ref{pseudopullback-theorem} tells that $\sC$ is
a $\kappa$\+accessible category, and provides the desired description
of its full subcategory of $\kappa$\+presentable objects.
\end{proof}

\begin{cor} \label{commutative-Noetherian-Gorenstein-flats-cor}
 Let $\kappa$~be an uncountable regular cardinal and $R$ be
a commutative Noetherian ring with the cardinality of
the spectrum\/ $|\Spec R|<\kappa$.
 Then the category\/ $\GF$ of Gorenstein-flat $R$\+modules is
$\kappa$\+accessible.
 The $\kappa$\+presentable objects of\/ $\GF$ are precisely
all the\/ $<\kappa$\+generated Gorenstein-flat $R$\+modules.
 So every Gorenstein-flat $R$\+module is a $\kappa$\+directed
colimit of\/ $<\kappa$\+generated Gorenstein-flat $R$\+modules.
\end{cor}

\begin{proof}
 Similar to the proof of
Corollary~\ref{dualizing-complex-Gorenstein-flats-cor} and based on
Theorem~\ref{commutative-Noetherian-F-totally-acyclics-theorem} together
with Proposition~\ref{saroch-stovicek-gorenstein-flat-basics}(a).
\end{proof}

 For a far-reaching generalization of
Corollary~\ref{commutative-Noetherian-Gorenstein-flats-cor}, see
Theorem~\ref{countably-coherent-Gorenstein-flats-theorem} below.

\begin{cor} \label{comm-Noetherian-F-totally-acyclics-direct-summands}
 Let $\kappa$~be an uncountable regular cardinal and $R$ be
a commutative Noetherian ring with the cardinality of
the spectrum\/ $|\Spec R|<\kappa$.
 Let $M$ be a\/ $<\nobreak\kappa$\+generated Gorenstein-flat $R$\+module.
 Then $M$ is a direct summand of an $R$\+module $N$ admitting
an F\+totally acyclic two-sided resolution $T^\bu$ with\/
$<\kappa$\+generated flat $R$\+modules $T^n$, \,$n\in\boZ$.
\end{cor}

\begin{proof}
 Similar to the proof of
Corollary~\ref{dualizing-complex-F-totally-acyclics-direct-summands}
and based on
Theorem~\ref{commutative-Noetherian-F-totally-acyclics-theorem}.
\end{proof}

 In particular, if the cardinality of $\Spec R$ is at most countable,
then Theorem~\ref{commutative-Noetherian-F-totally-acyclics-theorem}
and Corollaries~\ref{commutative-Noetherian-Gorenstein-flats-cor}\+-%
\ref{comm-Noetherian-F-totally-acyclics-direct-summands} are applicable
for $\kappa=\aleph_1$, providing a description of F\+totally acyclic
complexes of flat $R$\+modules as $\aleph_1$\+directed colimits of
F\+totally acyclic complexes of countably generated flat $R$\+modules.

\Section{Gorenstein-Flat Modules as Directed Colimits}
\label{gorenstein-flats-secn}

 In this section we use powerful and difficult results of
the paper~\cite{SarSt} in order to deduce a common generalization of
Corollaries~\ref{dualizing-complex-Gorenstein-flats-cor}
and~\ref{commutative-Noetherian-Gorenstein-flats-cor} applicable to
all right countably coherent rings~$R$.

\begin{lem} \label{epimorphism-from-directed-colimit-to-filtered}
 Let $R$ be an associative ring, $\kappa$~be a regular cardinal, and
$\lambda<\kappa$ be a smaller infinite cardinal.
 Let\/ $\sS$ and\/ $\sT$ be two classes of $\kappa$\+presentable
right $R$\+modules.
 Assume that the class of $R$\+modules\/ $\sT$ is closed under direct
summands, the class of $R$\+modules $\varinjlim_{(\kappa)}\sT$ is
closed under colimits of $\lambda$\+indexed chains in\/ $\Modr R$,
and the class\/ $\sS$ contains all the $\kappa$\+presentable
$R$\+modules belonging to\/ $\Fil(\sS)$.
 Let $F\in\varinjlim_{(\kappa)}\sT$ and $N\in\Fil(\sS)$ be two
$R$\+modules, and let $f\:F\rarrow N$ be a surjective $R$\+module
morphism.
 Then the morphism~$f$ is a $\kappa$\+directed colimit of surjective
$R$\+module morphisms $t\:T\rarrow S$ with $T\in\sT$ and $S\in\sS$.
\end{lem}

\begin{proof}
 This is another application of Theorem~\ref{pseudopullback-theorem}
together with Lemma~\ref{epimorphism-lemma} and the Hill lemma.
 Let $\sK_1$ be the category of epimorphisms of right $R$\+modules.
 Lemma~\ref{epimorphism-lemma} tells that the category $\sK_1$ is
$\kappa$\+accessible and provides a description of its full subcategory
of $\kappa$\+presentable objects.
 Let $\sL=\Modr R\times\Modr R$ be the Cartesian square of the abelian
category of right $R$\+modules.
 The category $\sL$ is locally $\kappa$\+presentable, and
Proposition~\ref{product-proposition} provides a description of its
full subcategory of $\kappa$\+presentable objects.

 Finally, let $\sK_2=\sF\times\sN$ be the Cartesian product of
the following two categories.
 The category $\sF$ is the class of modules
$\sF=\varinjlim_{(\kappa)}\sT$, viewed as a full subcategory
in $\Modr R$.
 To construct the category $\sN$, we need to recall the Hill
lemma~\cite[Theorem~6]{StT}, \cite[Theorem~7.10]{GT},
\cite[Theorem~2.1]{Sto}.
 In application to the $R$\+module $N$ with its given filtration by
modules from $\sS$, the Hill lemma produces a certain complete lattice
of submodules of~$N$.
 The category $\sN$ is this complete lattice of submodules in $N$,
viewed as a poset, and this poset is viewed as a category.
 (Notice that the category $\sN$ is \emph{not} additive.)

 As any complete lattice interpreted as a category, the category $\sN$
has all limits and colimits.
 In particular, the colimits in $\sN$ are just the joins in the lattice.
 Furthermore, in the situation at hand it follows from Hill lemma
that $\sN$ is a locally $\kappa$\+presentable category, and
the $\kappa$\+presentable objects of $\sN$ are precisely all
the $\kappa$\+presentable submodules of $N$ belonging to the lattice.
 On the other hand, the category $\sF$ is $\kappa$\+accessible, and
$\sT\subset\sF$ is its full subcategory of $\kappa$\+presentable
objects, by Proposition~\ref{accessible-subcategory}.
 Once again, Proposition~\ref{product-proposition} tells that
the category $\sK_2=\sF\times\sN$ is $\kappa$\+accessible, and
describes its full subcategory of $\kappa$\+presentable objects.

 The functor $\Phi_1\:\sK_1\rarrow\sL$ assigns to an epimorphism
of $R$\+modules $A\rarrow B$ the pair of $R$\+modules $(A,B)\in\sL$.
 The functor $\Phi_2\:\sK_2\rarrow\sL$ assigns to a pair of objects
$(G,M)$, where $G\in\sF$ and $M\in\sN$, the pair of $R$\+modules
$(G,M)\in\Modr R\times\Modr R=\sL$.
 In other words, $\Phi_2$ is the Cartesian product of two functors
$\sF\rarrow\Modr R$ and $\sN\rarrow\Modr R$.
 The functor $\sF\rarrow\Modr R$ is the identity inclusion
$\sF=\varinjlim_{(\kappa)}\sT\rarrow\Modr R$.
 The functor $\sN\rarrow\Modr R$ assigns the $R$\+module $M$ to
an element $M\subset N$ of the Hill lattice of submodules in~$N$.

 The pseudopullback category $\sC$ of the pair of functors $\Phi_1$
and $\Phi_2$ is the category of $R$\+module epimorphisms
$G\rarrow M$, where $G\in\sF$ and $M\subset N$, \,$M\in\sN$.
 In particular, $f\:F\rarrow N$ is an object of~$\sC$.
 Theorem~\ref{pseudopullback-theorem} tells that the category $\sC$
is $\kappa$\+accessible, and provides a description of its full
subcategory of $\kappa$\+presentable objects.
 It follows that the object $f\in\sC$ is a $\kappa$\+directed colimit
of $\kappa$\+presentable objects of $\sC$, as desired.
\end{proof}

\begin{thm} \label{countably-coherent-Gorenstein-flats-theorem}
 Let $R$ be a right countably coherent ring.
 Then the category of Gorenstein-flat right $R$\+modules\/ $\GF$ is\/
$\aleph_1$\+accessible.
 The\/ $\aleph_1$\+presentable objects of\/ $\GF$ are precisely
all the Gorenstein-flat right $R$\+modules that are countably
presentable in $\Modr R$.
 So every Gorenstein-flat right $R$\+module is an\/ $\aleph_1$\+directed
colimit of countably presentable Gorenstein-flat right $R$\+modules.
\end{thm}

\begin{proof}
 The full subcategory $\GF\subset\Modr R$ is closed under
directed colimits in $\Modr R$ by
Proposition~\ref{saroch-stovicek-gorenstein-flat-basics}(a).
 In view of Proposition~\ref{accessible-subcategory}, it remains to show
that any Gorenstein-flat right $R$\+module is an $\aleph_1$\+directed
colimit of countably presentable Gorenstein-flat right $R$\+modules.

 We use the description of Gorenstein-flat modules provided
by~\cite[Theorem~4.11(4)]{SarSt}, which tells that the Gorenstein-flat
modules are precisely all the kernels of surjective morphisms from flat
modules onto \emph{projectively coresolved Gorenstein-flat modules}.
 Furthermore, by~\cite[Theorem~4.9]{SarSt}, every projectively coresolved
Gorenstein-flat right $R$\+module is filtered by countably presentable
projectively coresolved Gorenstein-flat $R$\+modules.

 We apply Lemma~\ref{epimorphism-from-directed-colimit-to-filtered}
for $\kappa=\aleph_1$ and $\lambda=\aleph_0$.
 Let $\sT$ be the set of all countably presentable flat right
$R$\+modules, and let $\sS$ be the set of all countably presentable
projectively coresolved Gorenstein-flat right $R$\+modules.
 Then $\varinjlim_{(\aleph_1)}\sT$ is the class of all flat right
$R$\+modules by~\cite[Proposition~10.2]{Pacc}, while $\Fil(\sS)$ is
the class of all projectively coresolved Gorenstein-flat right
$R$\+modules.

 So the lemma is applicable, and it tells that every surjective morphism
from a flat right $R$\+module onto a projectively coresolved
Gorenstein-flat right $R$\+module is an $\aleph_1$\+directed colimit of
surjective morphisms from countably presentable flat $R$\+modules onto
countably presentable projectively coresolved Gorenstein-flat
$R$\+modules.
 It remains to pass to the kernels.  \hbadness=1050
\end{proof}

 Notice that Theorem~\ref{countably-coherent-Gorenstein-flats-theorem}
gives a better cardinality estimate for the accessibility rank of
the category of Gorenstein-flat $R$\+modules than
Proposition~\ref{deconstructible-as-directed-colimit} combined with
Proposition~\ref{saroch-stovicek-gorenstein-flat-basics}(b).

\Section{Modules of Finite Injective Dimension}
\label{finite-injective-dimension-secn}

 In this section we discuss the accessibility rank of the category of
modules of injective dimension~$\le m$, where $m$~is a fixed integer.
 In this context, our category-theoretic approach based on
Theorems~\ref{pseudopullback-theorem} and~\ref{isomorpher-theorem}
via Proposition~\ref{two-sided-by-accessible-classes-prop} gives
the same result as the deconstructibility-based approach of
Proposition~\ref{right-Ext-1-orthogonal-accessible}.

\begin{thm}  \label{finite-injective-coresolutions-theorem}
 Let $\lambda$ be a regular cardinal and $R$ be a right\/
$<\lambda$\+Noetherian ring.
 Let $\nu$~be a cardinal such that $\rho=|R|+\aleph_0\le\nu$
and $\nu^{<\lambda}=\nu$.
 Let $m\ge0$ be an integer.
 Then the category of all exact sequences of injective right
$R$\+modules $J^0\rarrow J^1\rarrow\dotsb\rarrow J^m\rarrow0$ is
$\nu^+$\+accessible.
 The $\nu^+$\+presentable objects of this category are precisely all
the exact sequences $T^0\rarrow T^1\rarrow\dotsb\rarrow T^m\rarrow0$
with injective right $R$\+modules $T^n$, \,$0\le n\le m$,
of cardinality at most~$\nu$.
 Consequently, every finite injective coresolution
$J^0\rarrow J^1\rarrow\dotsb\rarrow J^m\rarrow0$ in\/ $\Modr R$
is a $\nu^+$\+directed colimit of finite injective coresolutions
$T^0\rarrow T^1\rarrow\dotsb\rarrow T^m\rarrow0$ of the same length~$m$
with injective $R$\+modules $T^n$, \,$0\le n\le m$, of cardinality
at most~$\nu$.
\end{thm}

\begin{proof}
 Recall first of all that the full subcategory of injective $R$\+modules
is closed under $\lambda$\+directed colimits in $\Modr R$ by
Corollary~\ref{injectives-accessible}.
 The same corollary tells that the category of injective right
$R$\+modules is $\nu^+$\+accessible, and describes its full subcategory
of $\nu^+$\+presentable objects as consisting precisely of all
the injective $R$\+modules of cardinality at most~$\nu$.

 Now we apply Proposition~\ref{two-sided-by-accessible-classes-prop}
for $\kappa=\nu^+$ (notice that $\lambda\le\nu<\kappa$).
 Take $\sT_n$ to be the class of all injective right $R$\+modules of
cardinality at most~$\nu$ for all $0\le n\le m$, \
$\sT_{-1}$ to be the class of all right $R$\+modules of cardinality
at most~$\nu$, and $\sT_n=\{0\}$ for all $n\ge m+1$ and $n\le-2$.
 Then $\varinjlim_{(\kappa)}\sT_n$ is the class of all injective
right $R$\+modules for all $0\le n\le m$, \ $\varinjlim_{(\kappa)}
\sT_{-1}=\Modr R$, and $\varinjlim_{(\kappa)}\sT_n=0$ for all
$n\ge m+1$ and $n\le-2$.
 So the category $\sC$ from
Proposition~\ref{two-sided-by-accessible-classes-prop} is equivalent
to the category of finite injective coresolutions we are interested in.
\end{proof}

\begin{cor} \label{finite-injective-dimension-cor}
 Let $\lambda$ be a regular cardinal and $R$ be a right\/
$<\lambda$\+Noetherian ring.
 Let $\nu$~be a cardinal such that $\rho=|R|+\aleph_0\le\nu$
and $\nu^{<\lambda}=\nu$.
 Let $m\ge0$ be an integer.
 Denote by\/ $\sI_m\subset\Modr R$ the full subcategory of all
right $R$\+modules of injective dimension at most~$m$.
 Then the category\/ $\sI_m$ is $\nu^+$\+accessible.
 The $\nu^+$\+presentable objects of\/ $\sI_m$ are precisely all
the right $R$\+modules of cardinality~$\le\nu$ and
of injective dimension\/~$\le m$.
 So every right $R$\+module of injective dimension~$m$ is
a $\nu^+$\+directed colimit of right $R$\+modules of cardinality
at most~$\nu$ and of injective dimension at most~$m$.
\end{cor}

\begin{proof}
 A right $R$\+module $K$ has injective dimension~$\le m$ if and only
if $\Ext^{m+1}_R(R/I,K)\allowbreak=0$ for all right ideals $I\subset R$.
 Over a right $<\lambda$\+Noetherian ring $R$, the functor
$\Ext^*_R(M,{-})$ preserves $\lambda$\+directed colimits for every
$<\lambda$\+generated right $R$\+module~$M$.
 Consequently, the full subcategory $\sI_m$ is closed under
$\lambda$\+directed colimits in $\Modr R$.
 In view of Proposition~\ref{accessible-subcategory}, it remains to
check the last assertion of the corollary, which follows immediately
from Theorem~\ref{finite-injective-coresolutions-theorem}.  \hfuzz=3pt
\end{proof}

\begin{rem}
 The full subcategory $\sI_m\subset\Modr R$ can be also described as
$\sI_m=\sS^{\perp_1}$, where $\sS$ denotes the set of all
($<\lambda$\+generated) $m$\+th syzygy modules of $<\lambda$\+generated
(or just cyclic) right $R$\+modules.
 This allows to obtain the assertion of
Corollary~\ref{finite-injective-dimension-cor} directly as a particular
case of Proposition~\ref{right-Ext-1-orthogonal-accessible}.
\end{rem}

\Section{Injectively Resolved Modules}
\label{injectively-resolved-secn}

 By an \emph{injective resolution} we mean an exact sequence
$\dotsb\rarrow J_2\rarrow J_1\rarrow J_0$, where $J_n$ are injective
modules for all $n\ge0$.
 An $R$\+module $M$ is said to be \emph{injectively resolved} if
there exists an exact sequence of $R$\+modules
$\dotsb\rarrow J_2\rarrow J_1\rarrow J_0\rarrow M\rarrow0$
with injective $R$\+modules~$J_n$.

 In this section, as in the previous one, our category-theoretic
approach based on
Theorems~\ref{pseudopullback-theorem} and~\ref{isomorpher-theorem}
via Proposition~\ref{two-sided-by-accessible-classes-prop} produces
results which can be also obtained with the deconstructibility-based
approach using a suitable version of
Proposition~\ref{right-Ext-1-orthogonal-accessible}.

\begin{thm}  \label{injective-resolutions-theorem}
 Let $\lambda$ be a regular cardinal and $R$ be a right\/
$<\lambda$\+Noetherian ring.
 Let $\nu$~be a cardinal such that $\rho=|R|+\aleph_0\le\nu$
and $\nu^{<\lambda}=\nu$.
 Then the category of injective resolutions $\dotsb\rarrow J_2\rarrow
J_1\rarrow J_0$ in\/ $\Modr R$ is $\nu^+$\+accessible.
 The $\nu^+$\+presentable objects of this category are precisely all
the injective resolutions $\dotsb\rarrow T_2\rarrow T_1\rarrow T_0$
with injective right $R$\+modules $T_n$, \,$n\ge0$,
of cardinality at most~$\nu$.
 Consequently, every injective resolution $\dotsb\rarrow J_2\rarrow
J_1\rarrow J_0$ in\/ $\Modr R$ is a $\nu^+$\+directed colimit of
injective resolutions $\dotsb\rarrow T_2\rarrow T_1\rarrow T_0$
with injective $R$\+modules $T_n$, \,$n\ge0$, of cardinality
at most~$\nu$.
\end{thm}

\begin{proof}
 Similarly to the proof of
Theorem~\ref{finite-injective-coresolutions-theorem}, the argument in
based on Corollary~\ref{injectives-accessible} and
Proposition~\ref{two-sided-by-accessible-classes-prop}
for $\kappa=\nu^+$.
 Take $\sT_n$ to be the class of all injective right $R$\+modules of
cardinality at most~$\nu$ for all $n\le-1$, \
$\sT_0$ to be the class of all right $R$\+modules of cardinality
at most~$\nu$, and $\sT_n=\{0\}$ for all $n\ge1$.
 Then $\varinjlim_{(\kappa)}\sT_n$ is the class of all injective
right $R$\+modules for all $n\le-1$, \ $\varinjlim_{(\kappa)}
\sT_0=\Modr R$, and $\varinjlim_{(\kappa)}\sT_n=0$ for all $n\ge1$.
 So the category $\sC$ from
Proposition~\ref{two-sided-by-accessible-classes-prop} is equivalent
to the category of injective resolutions we are interested in.
\end{proof}

\begin{cor} \label{injectively-resolved-modules-cor}
 Let $\lambda$ be a regular cardinal and $R$ be a right\/
$<\lambda$\+Noetherian ring.
 Let $\nu$~be a cardinal such that $\rho=|R|+\aleph_0\le\nu$
and $\nu^{<\lambda}=\nu$.
 Then any injectively resolved right $R$\+module is a $\nu^+$\+directed
colimit of injectively resolved right $R$\+modules of cardinality
at most~$\nu$.
\end{cor}

\begin{proof}
 Let us emphasize that there is \emph{no} claim about the class of
injectively resolved modules being closed under $\mu$\+directed
colimits for any cardinal~$\mu$ in this corollary.
 The assertion of the corollary follows immediately from
Theorem~\ref{injective-resolutions-theorem}.
\end{proof}

\begin{cor} \label{injective-resolutions-direct-summands}
 Let $\lambda$ be a regular cardinal and $R$ be a right\/
$<\lambda$\+Noetherian ring.
 Let $\nu$~be a cardinal such that $\rho=|R|+\aleph_0\le\nu$
and $\nu^{<\lambda}=\nu$.
 Let $M$ be an injectively resolved right $R$\+module of
cardinality~$\le\nu$.
 Then $M$ is a direct summand of an $R$\+module $N$ admitting
an injective resolution $\dotsb\rarrow T_2\rarrow T_1\rarrow T_0$
with injective $R$\+modules $T_n$, \,$n\ge0$, of cardinality
at most~$\nu$.
\end{cor}

\begin{proof}
 This is a corollary of Theorem~\ref{injective-resolutions-theorem}
provable similarly to
Corollary~\ref{two-sided-by-lim-classes-direct-summand}.
 It is clear from the theorem that every injectively resolved
right $R$\+module $M$ is the colimit of a $\nu^+$\+directed diagram
of right $R$\+modules $(N_\xi)_{\xi\in\Xi}$ such that $N_\xi$ admits
an injective resolution $\dotsb\rarrow T_{2,\xi}\rarrow T_{1,\xi}
\rarrow T_{0,\xi}\rarrow N_\xi\rarrow0$ with injective $R$\+modules
$T_{n,\xi}$, \,$n\ge0$, of cardinality at most~$\nu$ for
every $\xi\in\Xi$.
 Now if $M$ has cardinality at most~$\nu$, then $M$ is
$\nu^+$\+presentable, and it follows that there exists $\xi\in\Xi$
such that $M$ is a direct summand of~$N_\xi$.
\end{proof}

\begin{rem}
 The full subcategory of acyclic complexes of injective modules in
$\Com(\Modr R)$ can be also described as the right $\Ext^1$\+orthogonal
full subcategory $\sS^{\perp_1}\subset\Com(\Modr R)$, where $\sS\subset
\Com(\Modr R)$ denotes the set of all contractible two-term complexes
of cyclic $R$\+modules $\dotsb\rarrow 0 \rarrow R/I\overset=\rarrow R/I
\rarrow0\rarrow\dotsb$ (placed in various cohomological degrees) and
one-term complexes of free $R$\+modules with one generator
$\dotsb\rarrow 0\rarrow R\rarrow0\rarrow\dotsb$ (placed in various
cohomological degrees).
 Here $I$~ranges over all the right ideals of $R$, and the construction
of the right $\Ext^1$\+orthogonal class of objects $\sS^{\perp_1}$
is applied in the abelian category of complexes $\Com(\Modr R)$.
 This assertion is a special case of~\cite[Proposition~4.6]{Gil}.
 Therefore, Theorem~\ref{injective-resolutions-theorem} can be also
obtained as a particular case of a suitable version of
Proposition~\ref{right-Ext-1-orthogonal-accessible} for the category
of complexes $\Com(\Modr R)$.
\end{rem}

\Section{Totally Acyclic Complexes of Injectives over a Left
Noetherian Ring} \label{left-noetherian-secn}

 An acyclic complex of injective left $R$\+modules $I^\bu$ is said to be
\emph{totally acyclic}~\cite[Section~5.2]{IK} if the complex of abelian
groups $\Hom_R(J,I^\bu)$ is acyclic for every injective left
$R$\+module~$J$.
 A left $R$\+module $M$ is said to be \emph{Gorenstein-injective}
if there exists a totally acyclic complex of injective left $R$\+modules
$I^\bu$ such that $M\simeq Z^0(I^\bu)$ is its module of cocycles.
 The definitions for right modules are similar.

 In the following four sections,
Sections~\ref{left-noetherian-secn}\+-\ref{full-generality-secn},
we study accessibility properties of the categories of totally acyclic
complexes of injective modules and of Gorenstein-injective modules.
 In the present section, the assumptions are more restrictive than
in Section~\ref{full-generality-secn}, but we obtain a better
cardinality estimate.

 We start with a lemma and a corollary suggested to the author by
Jan \v St\!'ov\'\i\v cek.

\begin{lem} \label{functorial-totally-injective-lemma}
 For any left Noetherian ring $R$, there exists a functor assigning to
every Gorenstein-injective left $R$\+module $M$ one of its totally
acyclic two-sided injective resolutions, i.~e., a totally acyclic
complex of injective left $R$\+modules $I^\bu$ together with
an isomorphism $M\simeq Z^0(I^\bu)$.
\end{lem}

\begin{proof}
 First of all we observe that, for any Gorenstein-injective
$R$\+module $M$ and an arbitrary injective $R$\+module coresolution
$0\rarrow M\rarrow I^0\rarrow I^1\rarrow I^2\rarrow\dotsb$ of $M$,
the complex $0\rarrow M\rarrow I^0\rarrow I^1\rarrow I^2\rarrow\dotsb$
remains exact after applying the functor $\Hom_R(J,{-})$ for any
injective $R$\+module~$J$.
 Furthermore, it is well-known that there exists a functor $M\longmapsto
(M\hookrightarrow I(M))$ assigning to an $R$\+module $M$ its
embedding into an injective $R$\+module $I(M)$.
 One can use a functorial version of the small object argument as
in~\cite[Theorem~2]{ET}, or simply let $I(M)$ be the direct product
of copies of the $R$\+module $\Hom_\boZ(R,\boQ/\boZ)$ indexed over
all the abelian group maps $M\rarrow\boQ/\boZ$.
 These observations allow to construct the positive cohomological degree
part of the desired functorial two-sided resolution of~$M$.
 This part of the argument does not use the Noetherianity assumption
on~$R$.

 The construction of the negative cohomological degree part of
the functorial two-sided resolution is based on the results
of~\cite[Lemma~2.1 and Example~6.4]{PS2}.
 By~\cite[Theorem~2.5]{Mat}, every injective left $R$\+module is a direct
sum of indecomposable injective $R$\+modules.
 It is easy to see that there is only a set of isomorphism classes of
indecomposable injectives.
 Therefore, there exists an injective left $R$\+module $J$ such that all
the injective left $R$\+modules are direct summands of direct sums of
copies of~$J$.
 Starting with a Gorenstein-injective left $R$\+module $M$, we denote
by $I^{-1}(M)$ the direct sum of copies of $J$ indexed over all
the $R$\+module maps $J\rarrow M$, that is
$I^{-1}(M)=J^{(\Hom_R(J,M))}$.
 Clearly, the natural $R$\+module map $I^{-1}(M)\rarrow M$ is
surjective (since $M$ is Gorenstein-injective).
 By~\cite[Lemma~2.1(c)]{PS2} (for $\sA=R\Modl$ and $T=J$), the kernel of
this map is a Gorenstein-injective left $R$\+module again, and so
the inductive process can be continued indefinitely, producing
the desired negative part of a totally acyclic two-sided injective
resolution in a functorial way.
\end{proof}

\begin{cor} \label{Noeth-Gorenst-injectives-rho-dir-colim-closed}
 For any left Noetherian ring $R$, the class of all Gorenstein-injective
left $R$\+modules is closed under $\rho^+$\+directed colimits in
$R\Modl$, where $\rho=|R|+\aleph_0$.
 Furthermore, the class of all totally acyclic complexes of injective
left $R$\+modules is closed under $\rho^+$\+directed colimits in\/
$\Com(R\Modl)$ as well.
\end{cor}

\begin{proof}
 Let $(M_\xi)_{\xi\in\Xi}$ be a diagram of Gorenstein-injective left
 $R$\+modules, indexed by some poset~$\Xi$.
 By Lemma~\ref{functorial-totally-injective-lemma}, one can construct
a $\Xi$\+indexed diagram of totally acyclic two-sided injective
resolutions $(I_\xi^\bu)_{\xi\in\Xi}$ for the given diagram of
$R$\+modules $(M_\xi)_{\xi\in\Xi}$.
 So the first assertion of the corollary follows from the second one.

 In order to prove the second assertion, notice first of all that
the class of injective left $R$\+modules is closed under directed
colimits in $R\Modl$ (e.~g., by the left-right opposite version of
Corollary~\ref{injectives-accessible} for $\lambda=\aleph_0$ and
$\nu=\rho$).
 Acyclicity of complexes of modules is obviously preserved by
directed colimits.

 To show that the class of totally acyclic complexes of injective left
$R$\+modules is closed under $\rho^+$\+directed colimits, notice that
in the definition of a totally acyclic complex of injectives it
suffices to check the preservation of acyclicity by the functors
$\Hom_R(J,{-})$, where $J$ ranges over \emph{indecomposable} injective
left $R$\+modules.
 Then it remains to observe that each indecomposable injective left
$R$\+module has cardinality at most~$\rho$ (again by
Corollary~\ref{injectives-accessible} for the same $\lambda$ and~$\nu$).
\end{proof}

 We refer to the recent preprint~\cite{Iac} for a general discussion of
directed colimit-closedness properties of Gorenstein-injective modules.
 For a version of
Corollary~\ref{Noeth-Gorenst-injectives-rho-dir-colim-closed} with more
restrictive assumptions and stronger conclusion, see
Corollary~\ref{dualizing-cmplx-Gorenst-injectives-dir-colim-closed}
below.

\begin{thm} \label{left-Noetherian-tot-acycl-of-inj-theorem}
 Let $R$ be a left Noetherian ring; put $\rho=|R|+\aleph_0$.
 Let $\nu$~be an infinite cardinal such that $\nu^\rho=\nu$.
 Then the category of totally acyclic complexes of injective left
$R$\+modules is $\nu^+$\+accessible.
 The $\nu^+$\+presentable objects of this category are precisely all
the totally acyclic complexes of injective left $R$\+modules of
cardinality at most~$\nu$.
 Consequently, every totally acyclic complex of injective left
$R$\+modules is a $\nu^+$\+directed colimit of totally acyclic complexes
of injective left $R$\+modules of cardinality at most~$\nu$.
\end{thm}

\begin{proof}
 This is an application of Theorem~\ref{pseudopullback-theorem}
together with Proposition~\ref{two-sided-by-accessible-classes-prop}.
 For a comparable or somewhat similar argument, see the proof
of Theorem~\ref{dualizing-complex-tot-acycl-of-inj-theorem} below.

 Let $\sK_1$ denote the category of acyclic complexes of injective
left $R$\+modules.
 By the left-right opposite versions of
Corollary~\ref{injectives-accessible} and
Proposition~\ref{two-sided-by-accessible-classes-prop} for
$\lambda=\aleph_0$, \,$\kappa=\nu^+$, and $\sT_n$ being the set of all
injective $R$\+modules of cardinality at most~$\nu$, for every
$n\in\boZ$, we know that $\sK_1$ is a $\nu^+$\+accessible category,
and have a description of its $\nu^+$\+accessible objects (cf.\
Theorem~\ref{injective-resolutions-theorem}).

 Furthermore, let $\sK_2$ denote the Cartesian product of copies of
the category of acyclic complexes of abelian groups, taken over
a set of representatives of isomorphism classes of indecomposable
injective left $R$\+modules~$J$.
 Notice that every indecomposable injective left $R$\+module is
an injective envelope of a cyclic $R$\+module $R/I$, where $I$ is
a left ideal in~$R$; and there is at most~$\rho$ of these.
 Clearly, $\rho<\nu<\nu^+$.
 Proposition~\ref{acyclic-complexes-prop} for $\kappa=\nu^+$ and
the ring of integers $\boZ$ in the role of $R$, together with
Proposition~\ref{product-proposition}, tell that $\sK_2$ is
a $\nu^+$\+accessible category and provide a description of its full
subcategory of $\nu^+$\+presentable objects.

 Finally, let $\sL$ be the Cartesian product of copies of
the abelian category of arbitrary complexes of abelian groups,
taken over the same set of indecomposable injectives~$J$.
 Lemma~\ref{abelian-complexes-lemma}(a,c) together with
Proposition~\ref{product-proposition} tell that $\sL$ is
a locally $\nu^+$\+presentable category and describe its full
subcategory of $\nu^+$\+presentable objects.

 Let $\Phi_1\:\sK_1\rarrow\sL$ be the functor taking any acyclic
complex of injective left $R$\+modules $I^\bu$ to the collection
of complexes of abelian groups $\Hom_R(J,I^\bu)$, indexed over
the indecomposable injective left $R$\+modules~$J$.
 Let $\Phi_2\:\sK_2\rarrow\sL$ be the Cartesian product of
the identity inclusions of the category of acyclic complexes of
abelian groups into the category of all complexes of abelian groups.
 Then the pseudopullback $\sC$ of the two functors $\Phi_1$
and $\Phi_2$ is the desired category of totally acyclic complexes
of injective left $R$\+modules.

 Recall that any indecomposable injective left $R$\+module has
cardinality at most~$\rho$ (as explained in the proof of
Corollary~\ref{Noeth-Gorenst-injectives-rho-dir-colim-closed}).
 Therefore, the functor $\Phi_1$ takes $\nu^+$\+presentable
objects to $\nu^+$\+presentable objects (since $\nu^\rho=\nu$)
and preserves $\rho^+$\+directed colimits.
 Put $\lambda=\rho^+$ and $\kappa=\nu^+$; then $\lambda<\kappa$.

 All the other assumptions of Theorem~\ref{pseudopullback-theorem}
are clearly satisfied.
 Theorem~\ref{pseudopullback-theorem} tells that $\sC$ is
a $\nu^+$\+accessible category, and provides the desired description
of its full subcategory of $\nu^+$\+presentable objects.
\end{proof}

\begin{cor} \label{left-Noetherian-Gorenstein-injectives-accessible}
 Let $R$ be a left Noetherian ring; put $\rho=|R|+\aleph_0$.
 Let $\nu$~be an infinite cardinal such that $\nu^\rho=\nu$.
 Then the category of Gorenstein-injective left $R$\+modules\/ $\GI$ is
$\nu^+$\+accessible.
 The $\nu^+$\+presentable objects of\/ $\GI$ are precisely all
the Gorenstein-injective left $R$\+modules of cardinality at
most~$\nu$.
 So every Gorenstein-injective left $R$\+module is a $\nu^+$\+directed
colimit of Gorenstein-injective left $R$\+modules of cardinality
at most~$\nu$.
\end{cor}

\begin{proof}
 The full subcategory $\GI\subset R\Modl$ is closed under
$\rho^+$\+directed colimits in $R\Modl$ by
Corollary~\ref{Noeth-Gorenst-injectives-rho-dir-colim-closed}.
 Since any $R$\+module of cardinality at most~$\nu$ is
$\nu^+$\+presentable in $R\Modl$, it follows that any
Gorenstein-injective left $R$\+module of cardinality at most~$\nu$
is $\nu^+$\+presentable in the category of Gorenstein-injective
left $R$\+modules.
 In view of Proposition~\ref{accessible-subcategory}, it remains to
check the last assertion of the corollary, which follows immediately
from Theorem~\ref{left-Noetherian-tot-acycl-of-inj-theorem}.
\end{proof}

 For a version of
Corollary~\ref{left-Noetherian-Gorenstein-injectives-accessible}
with more restrictive assumptions and a better cardinality estimate,
see Corollary~\ref{dualizing-complex-Gorenstein-injectives-accessible}
below.

\begin{cor} \label{left-Noetherian-tot-acycl-of-inj-direct-summand}
 Let $R$ be a left Noetherian ring; put $\rho=|R|+\aleph_0$.
 Let $\nu$~be an infinite cardinal such that $\nu^\rho=\nu$.
 Let $M$ be a Gorenstein-injective left $R$\+module of cardinality
at most~$\nu$.
 Then $M$ is a direct summand of an $R$\+module $N$ admitting
a two-sided totally acyclic injective resolution $T^\bu$ with injective
left $R$\+modules $T^n$, \,$n\in\boZ$, of cardinality at most~$\nu$.
\end{cor}

\begin{proof}
 This is a corollary of
Theorem~\ref{left-Noetherian-tot-acycl-of-inj-theorem}.
 The argument is similar to the proofs of
Corollaries~\ref{dualizing-complex-F-totally-acyclics-direct-summands}
and~\ref{injective-resolutions-direct-summands}.
\end{proof}

\Section{Dualizing Complexes and Totally Acyclic Complexes of
Injectives} \label{dualizing-and-totally-acyclic-of-injectives-secn}

 The present section prepares ground for the next one.
 As in Section~\ref{dualizing-and-F-totally-acyclic-secn}, we consider
two associative rings $R$ and $S$, and assume that the ring $S$ is
right coherent.
 We will say that a bounded complex of left $R$\+modules $K^\bu$ has
\emph{finite injective dimension} if it is isomorphic, as an object
of the derived category $\sD^\bb(R\Modl)$, to a bounded complex of
injective $R$\+modules.
 
 By a \emph{strong right dualizing complex of $R$\+$S$\+bimodules}
$D^\bu$ we mean a bounded complex of $R$\+$S$\+bimodules satisfying
conditions~(ii\+-iii) from
Section~\ref{dualizing-and-F-totally-acyclic-secn} together with
the following stronger version of condition~(i):
\begin{enumerate}
\renewcommand{\theenumi}{\roman{enumi}$'$}
\item the terms of the complex $D^\bu$ are fp\+injective as right
$S$\+modules, and the whole complex $D^\bu$ has finite injective
dimension as a complex of left $R$\+modules.
\end{enumerate}
 For a left Noetherian ring $R$, all fp\+injective left $R$\+modules
are injective, and there is no difference between right dualizing
complexes and strong right dualizing complexes.

 The following proposition is the main result of this section.
 It is dual-analogous to
Proposition~\ref{F-totally-acyclics-described-by-dualizing}.
 It is also another version of~\cite[Lemma~1.7]{Jor}.

\begin{prop} \label{tot-acycl-of-injectives-described-by-dualizing}
 Let $R$ be a ring, $S$ be a right coherent ring, and $D^\bu$ be
a strong right dualizing complex of $R$\+$S$\+bimodules.
 Then an acyclic complex of injective left $R$\+modules $I^\bu$ is
totally acyclic if and only if the complex of left $S$\+modules
$\Hom_R(D^\bu,I^\bu)$ is acyclic.
\end{prop}

 The proof of
Proposition~\ref{tot-acycl-of-injectives-described-by-dualizing}
is based on a sequence of lemmas.

\begin{lem} \label{Hom-from-bounded-acyclic-to-termwise-injective}
 Let $L^\bu$ be a bounded acyclic complex of left $R$\+modules and
$I^\bu$ be a complex of injective left $R$\+modules.
 Then the complex of abelian groups $\Hom_R(L^\bu,I^\bu)$ is acyclic.
\end{lem}

\begin{proof}
 This is dual-analogous to
Lemma~\ref{termwise-flat-tensored-with-bounded-acyclic}.
 The point is that the complex of abelian groups $\Hom_R(L^\bu,I^n)$
is acyclic for every $n\in\boZ$.
\end{proof}

\begin{cor} \label{Hom-from-bounded-injective-dimension}
 Let $I^\bu$ be a totally acyclic complex of injective left
$R$\+modules and $K^\bu$ be a bounded complex of left $R$\+modules.
 Assume that the complex $K^\bu$ has finite injective dimension.
 Then the complex of abelian groups $\Hom_R(K^\bu,I^\bu)$ is acyclic.
\end{cor}

\begin{proof}
 Follows from
Lemma~\ref{Hom-from-bounded-acyclic-to-termwise-injective}.
\end{proof}

 A left $R$\+module $C$ is said to be (\emph{Enochs})
\emph{cotorsion}~\cite[Section~2]{En2} if $\Ext^1_R(F,C)=0$ for
all flat left $R$\+modules~$F$.
 One can easily see that $\Ext^n_R(F,C)=0$ for all flat
$R$\+modules $F$, all cotorsion $R$\+modules $C$, and all $n\ge1$.

\begin{lem} \label{Hom-module-cotorsion}
 Let $R$ and $S$ be two rings, $E$ be an $R$\+$S$\+bimodule, and
$I$ be an injective left $R$\+module.
 Then the left $S$\+module $\Hom_R(E,I)$ is cotorsion.
\end{lem}

\begin{proof}
 One can immediately see that the functor $F\longmapsto
\Hom_S(F,\Hom_R(E,I))\simeq\Hom_R(E\ot_SF,\>I)$ takes short exact
sequences of flat left $S$\+modules $F$ to short exact sequences of
abelian groups.
\end{proof}

 The following lemma is a (much more nontrivial) dual-analogous
version of Lemma~\ref{acyclic-tensored-with-bounded-flat}.

\begin{lem} \label{Hom-from-bounded-flat-to-acyclic-cotorsion}
 Let $S$ be a ring, $C^\bu$ be an acyclic complex of cotorsion left
$S$\+modules, and $G^\bu$ be a bounded complex of flat left
$S$\+modules.
 Then the complex of abelian groups $\Hom_R(G^\bu,C^\bu)$ is acyclic.
\end{lem}

\begin{proof}
 The \emph{cotorsion periodicity theorem}~\cite[Theorem~5.1(2)]{BCE}
(see also~\cite[Corollary~8.2]{PS6} or~\cite[Theorem~5.3]{Pflcc})
claims that, in any acyclic complex of cotorsion modules, the modules
of cocycles are also cotorsion.
 Now it is clear that, for every $n\in\boZ$, the complex
$\Hom_R(G^n,C^\bu)$ is acyclic.
\end{proof}

\begin{proof}[Proof of
Proposition~\ref{tot-acycl-of-injectives-described-by-dualizing}]
 The ``only if'' implication is provided by
Corollary~\ref{Hom-from-bounded-injective-dimension}.
 To prove the ``if'', let $J$ be an injective $R$\+module.
 By Lemma~\ref{evaluation-quasi-isomorphism}, the natural morphism of
bounded complexes of left $R$\+modules $D^\bu\ot_S\Hom_R(D^\bu,J)
\rarrow J$ is a quasi-isomorphism.
 In view of Lemma~\ref{Hom-from-bounded-acyclic-to-termwise-injective},
in order to show that the complex $\Hom_R(J,I^\bu)$ is acyclic, it
suffices to check that so is the complex
$\Hom_R(D^\bu\ot_S\Hom_R(D^\bu,J),\>I^\bu)\simeq
\Hom_S(\Hom_R(D^\bu,J),\Hom_R(D^\bu,I^\bu))$.
 Now, by assumption, the complex of left $S$\+modules
$\Hom_R(D^\bu,I^\bu)$ is acyclic.
 By Lemma~\ref{Hom-module-cotorsion}, \,$\Hom_R(D^\bu,I^\bu)$ is
a complex of cotorsion left $S$\+modules.
 By Lemma~\ref{Hom-module-flat}, \,$\Hom_R(D^\bu,J)$ is a bounded
complex of flat left $S$\+modules.
 It remains to refer to
Lemma~\ref{Hom-from-bounded-flat-to-acyclic-cotorsion}.
\end{proof}

\begin{rem}
 Once again, the theory simplifies for rings that are Gorenstein in
a suitable sense.
 Similarly to Remark~\ref{gorenstein-flats-in-gorenstein-case}, one
can notice that, over any ring $R$, for any acyclic complex of injective
left $R$\+modules $I^\bu$ and any left $R$\+module $G$ of finite flat
dimension, the complex of abelian groups $\Hom_R(G,I^\bu)$ is acyclic
by Lemmas~\ref{Hom-from-bounded-acyclic-to-termwise-injective}
and~\ref{Hom-from-bounded-flat-to-acyclic-cotorsion}.
 Now if every injective left $R$\+module has finite flat dimension,
then it follows that every acyclic complex of injective left
$R$\+modules is totally acyclic, and every injectively resolved
left $R$\+module is Gorenstein-injective.
 In this case, the results of
Proposition~\ref{two-sided-by-accessible-classes-prop},
Theorem~\ref{injective-resolutions-theorem},
Corollary~\ref{injectively-resolved-modules-cor},
and Corollary~\ref{injective-resolutions-direct-summands}
become directly applicable (essentially) as properties of
totally acyclic complexes of injective modules and
Gorenstein-injective modules.
\end{rem}

\Section{Totally Acyclic Complexes of Injectives as Directed Colimits}
\label{totally-acyclics-of-injectives-as-dir-colims-secn}

 In the present section, the assumptions are somewhat restrictive, as
they include existence of a (strong) right dualizing complex; but
the cardinality estimates are even better than in
Section~\ref{left-noetherian-secn}.

 In addition to the conditions~(i$'$) and~(ii\+-iii) in the definition
of a strong right dualizing complex in
Sections~\ref{dualizing-and-F-totally-acyclic-secn}
and~\ref{dualizing-and-totally-acyclic-of-injectives-secn}, we will need
to assume another condition (which should be compared to condition~(iv)
from Section~\ref{F-totally-acyclics-as-dir-colims-secn}).
 We will say that a strong right dualizing complex of
$R$\+$S$\+bimodules $D^\bu$ has \emph{left type\/~$<\lambda$} if
\begin{enumerate}
\renewcommand{\theenumi}{\roman{enumi}$'$}
\setcounter{enumi}{3}
\item $D^\bu$ is isomorphic, as an object of the bounded derived
category of left $R$\+modules $\sD^\bb(R\Modl)$, to a bounded complex
of $<\lambda$\+generated left $R$\+mod\-ules~$\widetilde D^\bu$.
\end{enumerate}
 Condition~(iv$'$) will be used for left $<\lambda$\+Noetherian
rings $R$, so the classes of $<\nobreak\lambda$\+generated and
$\lambda$\+presentable left $R$\+modules coincide.
 In particular, in the case of the cardinal $\lambda=\aleph_0$, we will
speak about dualizing complexes of \emph{left finite type}.

\begin{thm} \label{dualizing-complex-tot-acycl-of-inj-theorem}
 Let $\lambda$ be a regular cardinal and $R$ be a left\/
$<\lambda$\+Noetherian ring.
 Let $\nu$~be a cardinal such that $\rho=|R|+\aleph_0\le\nu$
and $\nu^{<\lambda}=\nu$.
 Let $S$ be a right coherent ring and $D^\bu$ be a strong right
dualizing complex of $R$\+$S$\+bimodules.
 Assume that $D^\bu$ is a strong right dualizing complex of left
type\/~$<\lambda$.
 Then the category of totally acyclic complexes of injective left
$R$\+modules is $\nu^+$\+accessible.
 The $\nu^+$\+presentable objects of this category are precisely all
the totally acyclic complexes of injective left $R$\+modules of
cardinality at most~$\nu$.
 Consequently, every totally acyclic complex of injective left
$R$\+modules is a $\nu^+$\+directed colimit of totally acyclic complexes
of injective left $R$\+modules of cardinality at most~$\nu$.
\end{thm}

\begin{proof}
 This is an application of Theorem~\ref{pseudopullback-theorem}
together with Proposition~\ref{two-sided-by-accessible-classes-prop}
and Proposition~\ref{tot-acycl-of-injectives-described-by-dualizing}.
 The argument bears some similarity to the proof of
Theorem~\ref{left-Noetherian-tot-acycl-of-inj-theorem}.
 First of all, let us show that the full subcategory of totally acyclic
complexes of injective left $R$\+modules is closed under
$\lambda$\+directed colimits in the ambient abelian category
$\Com(R\Modl)$.

 By Corollary~\ref{injectives-accessible}, the full subcategory of
injective left $R$\+modules is closed under $\lambda$\+directed colimits
in $R\Modl$.
 It is also clear that acyclicity of complexes is preserved by all
directed colimits.
 Now let $I^\bu$ be an acyclic complex of injective left $R$\+modules,
and let $\widetilde D^\bu$ be a bounded complex of $<\lambda$\+generated
left $R$\+modules provided by the definition of a strong right dualizing 
complex of left type~$<\lambda$ (item~(iv$'$)) above.
 It is clear from
Lemma~\ref{Hom-from-bounded-acyclic-to-termwise-injective} that
the complex of left $S$\+modules $\Hom_R(D^\bu,I^\bu)$ is acyclic
if and only if the complex of abelian groups
$\Hom_R(\widetilde D^\bu,I^\bu)$ is acyclic.
 By Proposition~\ref{tot-acycl-of-injectives-described-by-dualizing},
we can conclude that the acyclic complex of injective $R$\+modules
$I^\bu$ is totally acyclic if and only if the complex of abelian groups
$\Hom_R(\widetilde D^\bu,I^\bu)$ is acyclic.
 The latter condition is obviously preserved by $\lambda$\+directed
colimits of the complexes~$I^\bu$.

 Now denote by $\sK_1$ the category of acyclic complexes of injective
left $R$\+modules.
 Applying the left-right opposite version of
Proposition~\ref{two-sided-by-accessible-classes-prop} for
$\kappa=\nu^+$ and $\sT_n$ being the set of all injective $R$\+modules
of cardinality at most~$\nu$, for every $n\in\boZ$, and using
Corollary~\ref{injectives-accessible} again, we see that $\sK_1$ is
a $\nu^+$\+accessible category, and obtain a description of its
$\nu^+$\+accessible objects (cf.\
Theorem~\ref{injective-resolutions-theorem}).
 Notice that $\lambda\le\nu$, so the ring $R$ is left
$<\nu^+$\+coherent.

 Furthermore, let $\sK_2$ denote the category of acyclic complexes of
abelian groups.
 Proposition~\ref{acyclic-complexes-prop} for $\kappa=\nu^+$ and
the ring of integers $\boZ$ in the role of $R$ tells that $\sK_2$ is
a $\nu^+$\+accessible category and provides a description of its full
subcategory of $\nu^+$\+presentable objects.
 Finally, let $\sL$ be the abelian category of arbitrary complexes of
abelian groups.
 Lemma~\ref{abelian-complexes-lemma}(a,c) tells that $\sL$ is
a locally $\nu^+$\+presentable category and describes its full
subcategory of $\nu^+$\+presentable objects as the category of
complexes of abelian groups of cardinality at most~$\nu$.

 Let the functor $\Phi_1\:\sK_1\rarrow\sL$ take any acyclic complex of
injective left $R$\+modules $I^\bu$ to the complex of abelian groups
$\Hom_R(\widetilde D^\bu,I^\bu)$.
 Let $\Phi_2\:\sK_2\rarrow\sL$ be the identity inclusion of
the category of acyclic complexes of abelian groups into the category
of all complexes of abelian groups.
 Then the pseudopullback $\sC$ of the two functors $\Phi_1$ and $\Phi_2$
is the category of all acyclic complexes of injective left $R$\+modules
$I^\bu$ for which the complex of abelian groups
$\Hom_R(\widetilde D^\bu,I^\bu)$ is acyclic.
 As we have seen in the second paragraph of this proof, the category
$\sC$ is the desired category of totally acyclic complexes of injective
left $R$\+modules.

 Clearly, the functor $\Phi_1$ takes $\nu^+$\+presentable objects to
$\nu^+$\+presentable objects (since $\nu^{<\lambda}=\nu$), and
preserves $\lambda$\+directed colimits.
 All the other assumptions of Theorem~\ref{pseudopullback-theorem}
(for $\kappa=\nu^+$) are also satisfied.
 Theorem~\ref{pseudopullback-theorem} tells that $\sC$ is
a $\nu^+$\+accessible category, and provides the desired description of
its full subcategory of $\nu^+$\+presentable objects.
\end{proof}

\begin{cor} \label{dualizing-cmplx-Gorenst-injectives-dir-colim-closed}
 Let $R$ be a left Noetherian ring, $S$ be a right coherent ring, and
$D^\bu$ be a right dualizing complex of $R$\+$S$\+bimodules.
 Assume that $D^\bu$ is a (strong) right dualizing complex of left
finite type.
 Then the class of all Gorenstein-injective left $R$\+modules is
closed under directed colimits in $R\Modl$.
 Furthermore, the class of all totally acyclic complexes of injective
left $R$\+modules is closed under directed colimits in\/
$\Com(R\Modl)$ as well.
\end{cor}

\begin{proof}
 In view of the first paragraph of the proof of
Corollary~\ref{Noeth-Gorenst-injectives-rho-dir-colim-closed}, it
suffices to prove the second assertion, which is a particular case
of the first two paragraphs of the proof of
Theorem~\ref{dualizing-complex-tot-acycl-of-inj-theorem}
(for $\lambda=\aleph_0$).
\end{proof}

 In the case of a commutative Noetherian ring $R$ with a dualizing
complex, the first assertion of
Corollary~\ref{dualizing-cmplx-Gorenst-injectives-dir-colim-closed}
can be obtained by combining~\cite[Theorem~2]{Iac}
with~\cite[Lemma~2.5(b)]{HJ}.

\begin{cor} \label{dualizing-complex-Gorenstein-inj-dir-colims-cor}
 Let $\lambda$ be a regular cardinal and $R$ be a left\/
$<\lambda$\+Noetherian ring.
 Let $\nu$~be a cardinal such that $\rho=|R|+\aleph_0\le\nu$
and $\nu^{<\lambda}=\nu$.
 Let $S$ be a right coherent ring and $D^\bu$ be a strong right
dualizing complex of $R$\+$S$\+bimodules.
 Assume that $D^\bu$ is a strong right dualizing complex of left
type\/~$<\lambda$.
 Then every Gorenstein-injective left $R$\+module is
a $\nu^+$\+directed colimit of Gorenstein-injective left $R$\+modules
of cardinality at most~$\nu$.
\end{cor}

\begin{proof}
 Follows immediately from the last assertion of
Theorem~\ref{dualizing-complex-tot-acycl-of-inj-theorem}.
\end{proof}

\begin{cor} \label{dualizing-complex-Gorenstein-injectives-accessible}
 Let $R$ be a left Noetherian ring; put $\rho=|R|+\aleph_0$.
 Let $S$ be a right coherent ring and $D^\bu$ be a right dualizing
complex of $R$\+$S$\+bimodules.
 Assume that $D^\bu$ is a (strong) right dualizing complex of left
finite type.
 Then the category of Gorenstein-injective left $R$\+modules\/ $\GI$ is
$\rho^+$\+accessible.
 The $\rho^+$\+presentable objects of\/ $\GI$ are precisely all
the Gorenstein-injective left $R$\+modules of cardinality at
most~$\rho$.
 So every Gorenstein-injective left $R$\+module is a $\rho^+$\+directed
colimit of Gorenstein-injective left $R$\+modules of cardinality
at most~$\rho$.
\end{cor}

\begin{proof}
 The full subcategory $\GI\subset R\Modl$ is closed under directed
colimits in $R\Modl$ by
Corollary~\ref{dualizing-cmplx-Gorenst-injectives-dir-colim-closed}.
 Since any $R$\+module of cardinality at most~$\rho$ is
$\rho^+$\+presentable in $R\Modl$, it follows that any
Gorenstein-injective left $R$\+module of cardinality at most~$\rho$
is $\rho^+$\+presentable in the category of Gorenstein-injective
left $R$\+modules.
 In view of Proposition~\ref{accessible-subcategory}, it remains to
check the last assertion of the corollary, which is provided by
Corollary~\ref{dualizing-complex-Gorenstein-inj-dir-colims-cor}
(for $\lambda=\aleph_0$ and $\nu=\rho$).
\end{proof}

\begin{cor} \label{dualizing-complex-tot-acycl-of-inj-direct-summand}
 Let $\lambda$ be a regular cardinal and $R$ be a left\/
$<\lambda$\+Noetherian ring.
 Let $\nu$~be a cardinal such that $\rho=|R|+\aleph_0\le\nu$
and $\nu^{<\lambda}=\nu$.
 Let $S$ be a right coherent ring and $D^\bu$ be a strong right
dualizing complex of $R$\+$S$\+bimodules.
 Assume that $D^\bu$ is a strong right dualizing complex of left
type\/~$<\lambda$.
 Let $M$ be a Gorenstein-injective left $R$\+module of cardinality
at most~$\nu$.
 Then $M$ is a direct summand of an $R$\+module $N$ admitting
a two-sided totally acyclic injective resolution $T^\bu$ with
injective left $R$\+modules $T^n$, \,$n\in\boZ$, of cardinality
at most~$\nu$.
\end{cor}

\begin{proof}
 This is a corollary of
Theorem~\ref{dualizing-complex-tot-acycl-of-inj-theorem}.
 The argument is similar to the proofs of
Corollaries~\ref{dualizing-complex-F-totally-acyclics-direct-summands}
and~\ref{injective-resolutions-direct-summands}.
\end{proof}

 In particular, if the ring $R$ in
Corollary~\ref{dualizing-complex-Gorenstein-injectives-accessible}
is (at most) countable, then $\rho^+=\aleph_1$; so every
Gorenstein-injective left $R$\+module is an $\aleph_1$\+directed colimit
of at most countable Gorenstein-injective left $R$\+modules.
 Furthermore, every Gorenstein-injective left $R$\+module of
cardinality~$\nu$ is a direct summand of an $R$\+module admitting
a two-sided totally acyclic injective resolution by injective
left $R$\+modules of cardinality at most~$\nu$.
 This holds for any infinite cardinal~$\nu$.

\Section{Totally Acyclic Complexes of Injectives in Full Generality}
\label{full-generality-secn}

 The results of this section, based on~\cite[Section~5]{SarSt}, are
an example of de\-con\-structibil\-ity-based approach to accessibility.
 They are applicable in much greater generality than the results of
Sections~\ref{left-noetherian-secn}
and~\ref{totally-acyclics-of-injectives-as-dir-colims-secn},
but the cardinality estimate is not as good.

\begin{prop} \label{totally-acyclic-injectives-accessible}
 Let $R$ be an associative ring; put $\rho=|R|+\aleph_0$.
 Let $\mu$~be an infinite cardinal such that $\mu^\rho=\mu$, and let
$\nu$~be an infinite cardinal such that $\nu^\mu=\nu$.
 Then the category of totally acyclic complexes of injective right
$R$\+modules is $\nu^+$\+accessible.
 The $\nu^+$\+presentable objects of this category are precisely all
the totally acyclic complexes of injective $R$\+modules of
cardinality at most~$\nu$.
 So every totally acyclic complex of injective $R$\+modules is
a $\nu^+$\+directed colimit (in fact, a $\nu^+$\+directed union)
of totally acyclic complexes of injective $R$\+modules of
cardinality at most~$\nu$.
\end{prop}

\begin{proof}
 By~\cite[Proposition~5.5]{SarSt}, there is a set $\sS$ of
$\mu^+$\+presentable objects in the abelian category of complexes
$\Com(\Modr R)$ such that the class of all totally acyclic complexes
of injective right $R$\+modules is the right $\Ext^1$\+orthogononal
class $\sS^{\perp_1}$ to $\sS$ in $\Com(\Modr R)$.
 The assertion of the proposition is provable by applying a suitable
version of Proposition~\ref{right-Ext-1-orthogonal-accessible}
for the category of complexes $\Com(\Modr R)$.
\end{proof}

\begin{prop} \label{Gorenstein-injectives-accessible}
 Let $R$ be an associative ring; put $\rho=|R|+\aleph_0$.
 Let $\mu$~be an infinite cardinal such that $\mu^\rho=\mu$, and let
$\nu$~be an infinite cardinal such that $\nu^\mu=\nu$.
 Then the category of Gorenstein-injective right $R$\+modules\/ $\GI$ is
$\nu^+$\+accessible.
 The $\nu^+$\+presentable objects of\/ $\GI$ are precisely all
the Gorenstein-injective $R$\+modules of cardinality at most~$\nu$.
 So every Gorenstein-injective $R$\+module is a $\nu^+$\+directed
colimit (in fact, a $\nu^+$\+directed union) of Gorenstein-injective
$R$\+modules of cardinality at most~$\nu$.
\end{prop}

\begin{proof}
 Once again, by~\cite[Theorem~5.6]{SarSt}, there is a set $\sS$ of
$\mu^+$\+presentable $R$\+modules such that $\GI=\sS^{\perp_1}
\subset\Modr R$ is the right $\Ext^1$\+orthogonal class to $\sS$ in
$\Modr R$.
 It remains to apply
Proposition~\ref{right-Ext-1-orthogonal-accessible}.
\end{proof}

\begin{cor} \label{totally-acyclic-injectives-direct-summand}
 Let $R$ be an associative ring; put $\rho=|R|+\aleph_0$.
 Let $\mu$~be an infinite cardinal such that $\mu^\rho=\mu$, and let
$\nu$~be an infinite cardinal such that $\nu^\mu=\nu$.
 Let $M$ be a Gorenstein-injective $R$\+module of cardinality
at most~$\nu$.
 Then $M$ is a direct summand of an $R$\+module $N$ admitting
a two-sided totally acyclic injective resolution $T^\bu$ with
injective $R$\+modules $T^n$, \,$n\in\boZ$, of cardinality
at most~$\nu$.
\end{cor}

\begin{proof}
 This is a corollary of
Proposition~\ref{totally-acyclic-injectives-accessible}.
 The argument is similar to the proofs of
Corollaries~\ref{dualizing-complex-F-totally-acyclics-direct-summands},
\ref{left-Noetherian-tot-acycl-of-inj-direct-summand},
and~\ref{dualizing-complex-tot-acycl-of-inj-direct-summand}.
\end{proof}

\end{document}